\documentclass[a4paper,10pt]{article} %CCNN 12pt
\usepackage{amsmath,amsthm,amssymb,enumerate}
\usepackage{xcolor}
\usepackage{mathrsfs}
\usepackage[affil-it]{authblk}
\usepackage{gensymb}
\usepackage[square,sort&compress,comma,numbers]{natbib}
\usepackage{bm}
\usepackage{hyperref}
\usepackage{subfig}
\usepackage{graphicx}
\usepackage[margin=3cm]{geometry}
\usepackage{tikz}
\usepackage{esint}
\usepackage{caption}
\usetikzlibrary{shapes,calc}
\usepackage{verbatim}
\usepackage{stmaryrd}
\usepackage{bm}
\usepackage{marginnote}

\newcommand{\mathbi}[1]{{\boldsymbol #1}}
\def\N{\mathbb{N}}

\newcommand{\ds}{{\rm\,ds}}

\def\R{\mathbb{R}}
\def\C{\mathbf{C}}
\def\dist{{\rm dist}}

\def\dx{\,{\rm d}\x}
\def\<{\langle}
\def\>{\rangle}
\def\Poly{\cP}

\def\bt{\begin{theorem}}
\def\et{\end{theorem}}
\def\el{\end{lemma}}
\def\bc{\begin{corollary}}
\def\ec{\end{corollary}}
\def\bd{\begin{definition}}
\def\ed{\end{definition}}
\def\br{\begin{remark}}
\def\er{\end{remark}}

\def\M{{\rm M}}
\def\NC{{\rm NC}}
\def\C{{\rm C}}
\def\err{{\rm err}}
\def\ndof{{\textbf{ndof}}}
\def\order{{\rm order}}

\DeclareMathOperator*{\argmin}{argmin}
\newcommand{\cA}{\mathcal A}
\newcommand{\cB}{\mathcal B}

\newcommand{\cT}{\mathcal T}

\newcommand{\cP}{\mathcal P}

\newcommand{\cQ}{\mathcal Q}

\newcommand{\mesh}{{\mathcal M}}

\newcommand{\bu}{\boldsymbol{u}}
\newcommand{\bF}{\boldsymbol{F}}
\newcommand{\bA}{\boldsymbol{A}}
\newcommand{\bl}{\boldsymbol{l}}
\newcommand{\bB}{\boldsymbol{B}}
\newcommand{\bfa}{\boldsymbol{f}}
\newcommand{\bW}{\boldsymbol{W}}
\newcommand{\bS}{\boldsymbol{S}}
\newcommand{\bbeta}{\boldsymbol{\beta}}
\newcommand{\bC}{\boldsymbol{C}}
\newcommand{\sit}{\sum_{K \in\mathcal{T}_h}\int_K}

\newcommand{\x}{\mathbi{x}}

\newcommand{\BFS}{Bogner-Fox-Schmit }

\newcommand{\be}{\begin{equation}}
\newcommand{\ee}{\end{equation}}

\renewcommand{\O}{\Omega}

\newcommand{\half}{{\frac 1 2}}

\newcommand{\ba}{\begin{array}{llll}   }
\newcommand{\bac}{\begin{array}{c}}
\newcommand{\bari}{\begin{array}{r}}
\newcommand{\ea}{\end{array}}

\newcommand{\NORM}[1]{{\left\vert\kern-0.25ex\left\vert\kern-0.25ex\left\vert #1 
    \right\vert\kern-0.25ex\right\vert\kern-0.25ex\right\vert}}

\def\hat#1{\widehat{#1}}

\newcommand{\fl}{\;\text{ for all }}
\def\t{\rm{ true}}
\def\a{{ \alpha}}

\usepackage{graphicx}
\usepackage{epstopdf}
\usepackage{pdfpages}
\usepackage{subfig,epsfig}
\usepackage{textcomp}
\usepackage{color}
\usepackage{tikz}
\usepackage{caption}

\usepackage{cancel}

\usetikzlibrary{arrows}
\usetikzlibrary{shapes,calc}

\usepackage[vlined]{algorithm2e}

\SetFuncSty{textsc}
\SetKwFunction{frun}{Run}
\SetKwHangingKw{arun}{\frun}
\SetKwFunction{fset}{Set}
\SetKwHangingKw{aset}{\fset}
\SetKwFunction{ffind}{Find}
\SetKwHangingKw{afind}{\ffind}
\SetKwFunction{fselect}{Select}
\SetKwHangingKw{aselect}{\fselect}
\SetKwFunction{fcompute}{Compute}
\SetKwHangingKw{acompute}{\fcompute}
\SetKwFunction{fsolve}{Solve}
\SetKwHangingKw{asolve}{\fsolve}
\SetKwFunction{fupdate}{Update}
\SetKwHangingKw{aupdate}{\fupdate}
\SetKwFunction{festimate}{Estimate}
\SetKwHangingKw{aestimate}{\festimate}
\SetKwFunction{fmark}{Mark}
\SetKwHangingKw{amark}{\fmark}
\SetKwFunction{fbreak}{Break}
\SetKwHangingKw{abreak}{\fbreak}
\SetKwFunction{frefine}{Refine}
\SetKwHangingKw{arefine}{\frefine}
\SetKwFunction{compute}{Compute}
\SetKwFunction{set}{Set}

\newenvironment{algof}%
{\algorithm}%
{\endalgorithm}% Figure/float layout

\usepackage[square,sort&compress,comma,numbers]{natbib}

\def\property#1{{\rm\textbf{(P#1)}}}

\newtheorem{theorem}{Theorem}[section]
\newtheorem{remark}[theorem]{Remark}

\newtheorem{lemma}[theorem]{Lemma} 
\newtheorem{definition}[theorem]{Definition}

\newtheorem{corollary}[theorem]{Corollary}
\newtheorem{assumption}[theorem]{Assumption}

\numberwithin{equation}{section}

\def\XXint#1#2#3{{\setbox0=\hbox{$#1{#2#3}{\int}$ }
\vcenter{\hbox{$#2#3$ }}\kern-.6\wd0}}

\usepackage{mathtools}  
\mathtoolsset{showonlyrefs} 
\usepackage[normalem]{ulem}
\newcounter{corr}
\definecolor{violet}{rgb}{0.580,0.,0.827}
\newcommand{\corr}[3]{\typeout{Warning : a correction remains in page
\thepage}
				\stepcounter{corr}        
				{\color{blue}\ifmmode\text{\sout{\ensuremath{#1}}}\else\sout{#1}\fi}
       {\color{red}#2}
       {\color{violet}#3}}

\newcounter{cexp}
\catcode`\@=11
\def\terml#1{T_{\refstepcounter{cexp}\@bsphack
\protected@write\@auxout{}%
           {\string\newlabel{#1}{{\thecexp}{\thepage}}}\thecexp}}
\catcode`\@=12

\date{}
\begin{document}
\title{Conforming and Nonconforming Finite Element Methods for Biharmonic Inverse Source Problem}%[FEMs for Biharmonic Inverse Source Problem]
\author{Devika Shylaja,\, M. T. Nair\footnote{Department of Mathematics, Indian Institute of Technology Madras, Chennai, Tamil Nadu 600 036, India.
	011353@imail.iitm.ac.in; {mtnair@iitm.ac.in}}}%\footnote{Department of Mathematics, Indian Institute of Technology Madras, Chennai, Tamil Nadu 600 036, India.{011353@imail.iitm.ac.in}}
\maketitle

\begin{abstract}
	This paper deals with the numerical approximation of the biharmonic inverse source problem in an abstract setting in which the measurement data is finite-dimensional. This unified framework in particular covers the conforming and nonconforming finite element methods (FEMs). The inverse problem is analysed through the forward problem. Error estimate for the forward solution is derived in an abstract set-up that applies to conforming and Morley nonconforming FEMs. Since the inverse problem is ill-posed, Tikhonov regularisation is considered to obtain a stable approximate solution. Error estimate is established for the regularised solution for different regularisation schemes. Numerical results that confirm the theoretical results are also presented.
\end{abstract}
{\bf Keywords: } finite element methods, biharmonic problem, inverse source problem, error estimates, Tikhonov regularization
\section{Introduction}\label{sec.intro}
Let $\Omega\subset\mathbb{R}^2$ be a domain with Lipschitz boundary. Consider the boundary value problem associated with the biharmonic equation: Given the source field $f$, find the displacement field $u$ such that
\begin{subequations}\label{biharmoniceq.bdy}
	\begin{align}
	&  \Delta^2  u = f \mbox{ in } \Omega, \label{biharmoniceq} \\
	&u=0,\,\frac{\partial u}{\partial n}=0
	\text{  on }\partial\Omega, \label{bdycondition}
	\end{align}
\end{subequations}
 where $\Delta^2$ denotes the fourth order biharmonic operator given by $$\Delta^2u=u_{xxxx}+u_{yyyy}+2u_{xxyy}$$  and $n$ denotes the outward normal vector to the boundary $\partial\Omega$ of $\Omega$. The boundary value problem \eqref{biharmoniceq.bdy} describes the bending of a thin elastic plate which is clamped along the boundary and acted upon by the vertical force $f$ \cite{ciarlet1978finite}. The biharmonic equations arise in various applications, for example, in applied mechanics, thin plate theories of elasticity and the Stokes problem in stream function and vorticity formulation \cite{ciarlet_plate,NS_FE79}. Several schemes, such as conforming and nonconforming finite element methods \cite{BS94,dou_percell_scott,PLPL75,HDM_linear,DS_HDM}, $C^0$ interior penalty methods \cite{SBLS05} and discontinuous Galerkin methods \cite{dG_IMES}, have been discussed in literature for the numerical approximation of biharmonic equation. %Conforming finite element methods for \eqref{biharmoniceq.bdy} requires the approximation space to be a subspace of $H^2_0(\O)$, which results in $C^1$ finite elements. One of the main challenges in the implementation while approximating solutions of fourth order problems using conforming finite elements is that the corresponding  strong continuity requirement of function and its derivatives makes it difficult to construct such a finite element, and hence one needs to handle with 21 degrees of freedom in a triangle and 16 degrees of freedom in a rectangle.
 
 \medskip
 
 In certain situations, one may need to determine the source functions with
 some information about the solution $u$. These type of problems are clearly inverse
 to the direct problem of finding a solution to the partial differential equations known as the inverse problems. An application of such a problem associated with the biharmonic equations is to determine how much force to be applied for bending of a thin elastic plate to a particular  position. Such inverse problems are ill-posed as the solution does not depend continuously on the data. For example, consider
 $$u_k(x,y)=\frac{\sin^2(k\pi x)\sin^2(k\pi y)}{k^3} \quad x,y \in [0,1]^2, k \in \N.$$
 % It can be verified that $u_k \in H^2_0(\O)$ and the sequence of data $(u_k)$ converges to $0$ as $k \to \infty$ whereas the corresponding sequence of solutions $(f_k)$ diverges.
 We observe that each $u_k$ satisfies \eqref{biharmoniceq}-\eqref{bdycondition} with 
 \begin{align*}
 f_k&= 8k\pi^4\cos^2(k\pi x )(\cos^2(k \pi y)-2\sin^2(k \pi y))\\
 &\qquad -8k\pi^4\sin^2(k\pi x )(2\cos^2(k \pi y)-3\sin^2(k \pi y))
 \end{align*}
 in place of $f$. Note that 
 $$\|u_k \|_{H^2(\O)} \leq \frac{1}{k}$$
 whereas
 $$\| f_k \|_{L^2(\O)} \geq 16\pi^8k^2.$$
 Thus,the sequence of data $(u_k)_{k \in \N}$ converges to $0$ in $H^2_0(\O)$ while the corresponding sequence of solutions $(f_k)_{k \in \N}$ diverges in $L^2(\O)$. 
 
 \medskip
 
 The paper deals with the numerical analysis of the biharmonic inverse source problem, more precisely, to determine $f$ provided $u$ is given only at a finite number of locations in $\O$. A finite number of measurements by themselves can not uniquely determine the approximation to the source field $f$ in a stable manner.  Hence, we consider reconstructions obtained using Tikhonov regularization \cite{IP_HEMHAN, MTN_book_LOE} where the ill-posed problem is replaced with a nearby well-posed problem which is an approximation to the original problem. The regularized problem  is then discretised and an error estimate is derived in an abstract setting for different regularization schemes, for instance, $L^2, H^1$ and $H^2$ regularization. This setting is shown to cover the conforming and nonconforming finite element methods (FEMs). This work is motivated from \cite{IP_AHSBAH} where the authors derive the error estimates for the Poisson inverse source problem that employs the conforming FEM.
 
 \medskip
 
The inverse problem is analysed through the forward direction of the problem \eqref{biharmoniceq.bdy}. Conforming FEMs for \eqref{biharmoniceq.bdy} requires the approximation space to be a subspace of $H^2_0(\O)$, which results in $C^1$ finite elements. One of the main challenges in the implementation while approximating solutions of fourth order problems using conforming finite elements is that the corresponding strong continuity requirement of function and its derivatives makes it difficult to construct such a finite element, and hence one needs to handle with 21 degrees of freedom in a triangle (for Argyris FEM) and 16 degrees of freedom in a rectangle (for \BFS FEM) \cite{ciarlet1978finite}. The nonconforming FEM relaxes the $C^1$ continuity requirement of the finite element space. The advantage of Morley nonconforming FEM is that it uses piecewise quadratic polynomials for the approximation and hence is simpler to implement. However, the convergence analysis offers a lot of challenges since the discrete space is not a subspace of $H^2_0(\O)$. In this paper, error estimate is established for the solution of the forward problem in a generic framework that is useful to prove the error estimates for the  inverse problem. % It seems that the Bogner-Fox-Schmit rectangle (1965) (see \cite{ciarlet1978finite, BFS}) is the only $C^1$ element on rectangular grids which has comparatively lesser number of degrees of freedom though it has limited applicability. 
 
 \medskip

The main contributions of this article are summarized as follows:
	\begin{itemize}
 	 {\it	\item Error estimate of the forward problem using measurment function when displacement $u$ is approximated in an abstract setting under a few generic assumptions;
 		\item Application to the \BFS and Argyris conforming FEMs, and Morley nonconforming FEMs;
 		\item Error estimate of the inverse problem in $H^k$ norm ($k \in \N_0=\N \cup \{0\}$) when $f \in H^k(\O)$ and the reconstructed regularised approximation of the source field is approximated in a unified framework under a few generic assumptions;
 	\item Application to the conforming and non-conforming FEMs;  
 	\item Numerical implementation procedure and results of computational experiments that validate the theoretical estimates.}
 \end{itemize}

\medskip

The rest of the paper is organized as follows. Section \ref{sec.forwardproblem} presents the forward problem, its weak formulation and some auxiliary results relevant for the paper. Section \ref{sec.abstract.forward} proves the abstract results for the forward problem under certain assumptions. The measurement model is discussed and is followed by its finite element approximation in Section \ref{sec.measurement}. Error estimate of the forward problem using these models that is useful to analyse the inverse problem is derived at the end of this section. Application to \BFS and Argyis conforming FEMs, and Morley nonconforming FEM are discussed in Section \ref{sec:fem}. Section \ref{sec.inverseproblem} deals with the inverse problem and the reconstructed source field obtained as a Tikhonov regularized inverse of the forward problem. The reconstructed regularised approximation of source field is further discretised using a piecewise polynomial finite element space and error estimates are derived in an abstract setting. This setting then applies to the conforming and nonconforming FEMs in Section \ref{sec.feminverse}. Section \ref{sec.numericalsection} deals with the description of a numerical implementation procedure and the results of the numerical experiments for $k=0,1,2$ that support the theoretical estimates obtained in the previous sections.

 \medskip
 
 Throughout the paper, standard notations on Lebesgue and
 Sobolev spaces and their norms are employed. The standard semi-norm and norm on $H^{s}(\Omega)$ (resp. $W^{s,p} (\Omega)$) for $s>0$ and $1 \le p \le \infty$ are denoted by $|\cdot|_{s}$ and $\|\cdot\|_{s}$ (resp. $|\cdot|_{s,p}$ and $\|\cdot\|_{s,p}$ ). The standard $L^2$ inner product and norm are denoted by $(\cdot, \cdot)$ and $\|\cdot\|,$ respectively. Denote by $H^{-s}(\Omega)$ the dual space of $H^s_0(\Omega)$ equipped with the norm
 \begin{equation*}
\|f\|_{-s}= \|f\|_{H^{-s}(\O)}:=\sup\left\{\frac{\langle f,g\rangle}{\|g\|_s}: g\in H^s_0(\Omega),\;\|g\|_s\neq 0\right\},
 \end{equation*}
 where $\langle \cdot,\cdot \rangle $ denotes duality pairing between $H^s_0(\Omega)$ and $H^{-s}(\Omega)$. Let
 $$V:=H^2_0(\O)={}\left\{ v \in H^2(\O); v=\frac{\partial v}{\partial n} = 0 \mbox{ on } \partial\O \right\}\mbox{ and } V':=H^{-2}(\O).$$
 We may recall that $V$ is a Hilbert space with associated norm $\|\Delta\cdot\|$.
 %Note that $(V, \| \Delta\cdot \|)$ is a Hibert space.
%\smallskip
%
%The notation $a\lesssim b$ (resp. $a \gtrsim  b$) means there exists a generic mesh independent constant $C$ such that $a\leq Cb$ (resp. $a\ge Cb$). 
 %The positive constants $C$ appearing in the inequalities denote generic constants which do not depend on the mesh-size. 
% Circles and rectangles are Lipschitz domains, but a circle with a cut is not.All C^k domains (k \ge1) are Lipschitz

\section{Forward problem}\label{sec.forwardproblem}
 The weak formulation of the forward problem \eqref{biharmoniceq.bdy} seeks $u \in V$ corresponding to a given $f \in V'$ such that
 \begin{equation}\label{weak.forward}
 a(u,v)=\langle f,v \rangle \quad \forall\, v \in V,
 \end{equation}
 where the bilinear form $a(\cdot,\cdot):V\times V\rightarrow \mathbb{R}$ is defined by
 \begin{align}\label{bilinear.a}
& a(v,w):=\int_\O \Delta v \Delta w \dx.
 \end{align}
 %using the homogeneous boundary conditions \eqref{bdycondition}. Here, $\hessian$ is the Hessian matrix and $A:B$ denote the tensor product of the matrices $A$ and $B$. Note that $\|\Delta v\|=\| D^2 v\|$.
  The properties regarding the boundedness and coercivity of $a(\cdot,\cdot)$ in $V$ can be easily verified and are stated now: %for all $v, w \in V$,
  \begin{itemize}
  	\item continuity: there exists a constant $C_b>0$ such that
  	$$ a(v,w) \le C_b\|v\|_2\|w\|_2 \quad \forall\, v, w \in V.$$
  	\item coercivity: there exists a constant $C_c>0$ such that
  	$$a(v,v)\ge C_c\| v\|_{2}^2 \quad \forall\, v \in V.$$
  \end{itemize}
 %$$ a(v,w) \le \|v\|_2\|w\|_2, \quad a(v,v)\ge \|v\|_2^2.$$
% \begin{itemize}
% 	\item  $a(\cdot,\cdot)$ is continuous and coercive in the Hilbert space $V$.
% 	\item $v \rightarrow \langle f,v\rangle$ is a continuous linear functional on $V$ for each $f \in V'$. 
% 	\end{itemize}
 Hence, by the Lax-Milgram Lemma \cite{evans2010,Kesavan}, there exists a unique solution $u \in V$ to \eqref{weak.forward}. Moreover, the solution $u$ satisfies %Also, from the Cauchy-Schwarz inequality and Poincar\'e inequality, it holds that
 \begin{equation}\label{stability}
 \| u\|_2 \le C\|f \|_{-2},
 \end{equation}
 where $C>0$ is independent of $f$. 
 
 \begin{remark}
 %For $f \in L^2(\O)$, the existence and uniqueness of the weak solution $u \in V$ to \eqref{weak.forward}  follows from the Lax-Milgram Lemma.
 If $f \in L^2(\O)$, then $f$ can be considered as an element of $V'$ and in that case, we only need to replace the right hand side of \eqref{weak.forward} by $(f,v)$ for all $v \in V$. Also, from \eqref{stability}, we have $\|u\|_2 \le C\|f\|$, as $\|\cdot\|_{-2}$ is weaker than $\| \cdot \|$.%, where $C_p$ is the Poincar\'e constant of $\O$.
 \end{remark} 
 \medskip
 
 Let $T$ denotes the solution operator of weak formulation of the biharmonic problem \eqref{weak.forward}. That is, $T: V' \rightarrow L^2(\O)$ with its range contained in $V$ is such that 
 \begin{equation}\label{operatorT}
 Tf=u,
 \end{equation}
where $u$ is the solution to \eqref{weak.forward} with the source term $f$. Therefore, given $f \in V'$, \eqref{weak.forward} can be rewritten as
 \begin{equation}\label{weak.forwardoperator}
 a(Tf,v)=\langle f, v\rangle \quad \forall\, v \in V.
 \end{equation}
% Since $f \in L^2(\O)$, $v \rightarrow (f,v)$ defines a bounded linear functionals on $H^1_0(\O)$. Also,
% \begin{equation*}
% (f,v)\le \| f \|\| v\| \le C_p \|f\| \|v\|_1,
% \end{equation*}
% where $C_p$ is the Poincar\'e constant of $\O$. Hence,
% \begin{equation}\label{f}
% \| f \|_{-1}\le C_p\| f \|,
% \end{equation}
% 
% \medskip
 
 Since the $H^{-2}(\Omega)$ norm on $L^2(\Omega)$ is weaker than $L^2(\Omega)$ norm and the imbedding of $V$ into $L^2(\Omega)$ is compact, it can be shown that $T$ is a compact operator from $L^2(\Omega)$ into itself.  Further, $T:  L^2(\Omega)\to L^2(\Omega)$ is a positive and self adjoint operator with non-closed range.
 
 \smallskip
 
  The minimal elliptic regularity result related to biharmonic equation is stated in the next lemma. 
 \begin{lemma}\cite[Theorem 2]{HBRR}\label{bih_reg_res} Let $\Omega$ be a bounded polygonal domain and for $f \in V'$, let $u \in V$ be the unique solution to \eqref{weak.forward}. If $f \in L^{2}(\O)$, then $u \in H^{2+\gamma}(\Omega)$ and there exists a constant $C>0$ independent of $f$ such that
 	\begin{equation*}
 	\|u\|_{{2+\gamma}}\leq C\|f\|_{{}},
 	\end{equation*}
 	where  $\gamma \in(\half,1]$ is the index of elliptic regularity determined by the interior angles at the corner of the domain $\Omega$ and $\gamma=1$ when $\Omega$ is convex.
 	
 	\smallskip
 	
 	In additon, if $\Omega$ is convex with all the interior angles of $\Omega$ are less than $126.283^\circ$, then $u \in H^4(\Omega)$ and $\|u\|_{{4}}\leq C\|f\|$.	
 \end{lemma}
Also, the following elliptic regularity result is well explained in \cite[Theorem 2.20]{PolyharmonicBVP_2010} and \cite[Theorem 15.2]{1959_SAADLN}.
\begin{lemma}\label{regularity.generic}
Assume that $\partial{\O}\in \mathcal{C}^m, m\ge 4$. Then for all $f \in H^{m-4}(\O)$, \eqref{weak.forward} admits a unique solution $u \in H^m(\O)\cap H^2_0(\O)$; moreover there exists a constant $C$ independent of $f$ such that $\|u\|_{m} \le C\|f\|_{m-4}.$
\end{lemma}
We make the following assumption based on the above two lemmas which talks about the regularity result of the solution of biharmonic problem.
\begin{assumption}\label{assumption.regularity.forward}
For a given $k \in \N_0=\N \cup \{0\}$, there exists an $s\ge 2$ such that	for every $f \in H^k(\O)$, the solution $u$ to \eqref{weak.forward} is in $H^s(\O)$ and $\|u\|_s \le C\|f\|_k$ where $C>0$ is independent of $f$.
\end{assumption}
Throughout the paper, the symbols $k$ and $s$ indicate the relationship between $f$ and $u$ given by the above assumption.

\subsection{Discrete formulation and Abstract results}\label{sec.abstract.forward}
This section deals with the discrete formulation of \eqref{weak.forward} and a unified convergence analysis.

\smallskip

Let $V_h$ be a finite element space in which the approximate solution to \eqref{weak.forward} is sought. Here, $\O$ is partitioned into elements (e.g. triangles/rectangles) and the approximation space $V_h$ is made of piecewise polynomials of degree atmost $p$ on this partition.

The finite element  formulation corresponding to \eqref{weak.forward} seeks $u_{h} \in V_h$ such that
\begin{align}\label{weak.forward.fem}
& a_h(u_h,v_h)=(f,v_h) \quad \forall\, v_h \in V_h,
\end{align}
where $a_h:(V+V_h) \times (V+V_h) \to \R$ is a bounded linear form with norm $\|\cdot\|_h$. Assume that $a_h(\cdot,\cdot)$ is coercive in $V_h$ with respect to $\|\cdot\|_h$. Thus the discrete problem \eqref{weak.forward.fem} is well-posed and it holds
\begin{equation}\label{stability.fem}
\| u_h\|_h \le C\|f \|,
\end{equation}
where $C>0$ is independent of $f$ and $h$. The choice of $a_h(\cdot,\cdot)$ and $\|\cdot\|_h$ for conforming and nonconforming FEMs are described in Section \ref{sec:fem}.

%for all $\eta,\chi \in V+V_h,$
%$$\displaystyle a_h(\eta,\chi):=\sit D^2 \eta : D^2\chi \dx$$
%with $D^2\cdot$ denotes the Hessian of the argument and $:$ denotes the scalar product of matrices.
%%$a_h(\cdot,\cdot)$ is defined by
%%$$\displaystyle a_h(\eta_h,\chi_h):=\sit D^2 \eta_h:D^2\chi_h \dx \quad \forall\, \eta_h,\chi_h \in V_ {h}.$$
%
%\medskip
%
%Define the energy norm $\|\cdot\|_h:=\big(a_h(\cdot,\cdot)\big)^{1/2}$. Since $a_h(\cdot,\cdot)$ is symmetric and positive definite on the set $V_h$, the discrete problem \eqref{weak.forward.fem} is well-posed and it holds that
%\begin{equation}\label{stability.fem}
%\| u_h\|_h \le C\|f \|,
%\end{equation}
%where $C>0$ is independent of $f$ and $h$.

\medskip

In the finite dimensional setting, let $T_h$ denotes the finite element solution operator of the biharmonic problem. That is, $T_h: L^2(\O) \rightarrow L^2(\O)$ with its range contained in $V_h$ is such that 
\begin{equation}\label{operatorTh}
T_h f=u_h,
\end{equation}
where $u_h$ is the solution to \eqref{weak.forward.fem} with source $f \in L^2(\O)$. Therefore, \eqref{weak.forward.fem} can be rewritten as
\begin{equation}\label{weak.forwardoperator.ncfem}
a_h(T_h f,v_h)=(f,v_h) \quad \forall\, v_h \in V_h.
\end{equation}
 In the following, the notation $a\lesssim b$ means there exists a generic mesh and source independent constant $C$ such that $a\leq Cb$ unless otherwise specified. %. %The positive constants $C$ appearing in the inequalities denote generic constants which do not depend on the mesh-size $h$ and source $f$  
%From now onwards, $C>0$ denotes a generic constant which take different values at different contexts, but independent of $f$ and $h$.

\medskip

We prove the error of the forward problem in an abstract way under the following assumption.
\begin{assumption}\label{assumption.forwardoperators}
There exists {\it interpolation operator} $I_h \in L(V,V_h)$ and {\it enrichment operator} $E_h \in L(V_h,V)$ with following properties that lead to non-negative parameters $\delta_1,\delta_2,\delta_3,\delta_4$ and $\delta_5$.
\begin{itemize}
	\item[\property{1}]
For $v\in V\cap H^{s}(\Omega)$, $\sup_{\|v\|_s=1}\|v-I_hv\|_{h}=\delta_1(s).$

\item[\property{2}] For any $v_{h}, w_h \in V_h$ and $\chi \in V \cap H^s(\O)$,
\begin{itemize}
\item [$(a)$]
$ \| v_{h}-E_hv_{h}\|\le \delta_2\dist(v_h,V).$
\item [$(b)$] $\|v_{h}-E_hv_{h}\|_h\le \delta_3\dist(v_h,V). $
\item[$(c)$] $\sup_{\|v_h\|_h=1}\sup_{\|\chi\|_s=1}|a_{h}(v_h-E_{h}v_{h},\chi)|= \delta_4.$
\item[$(d)$] $\sup_{\|v_h\|_h=1}\sup_{\|w_h\|_h=1}|a_h(v_h-E_hv_h,w_h)|= \delta_5.$
\end{itemize}
\end{itemize}	
\end{assumption}
For instance, for conforming FEMs, there exists an interpolation operator $I_h$ with $\delta_1(s)\lesssim h^{\min\{p+1,\;s\}-2}$ and the enrichment operator $E_h$ is the identity operator $\rm{Id}$. Hence, in this case, the properties \property{1} and \property{2}$(a)-(d)$ are satisfied. We will discuss these properties in details for conforming and nonconforming FEMs in Section \ref{sec:fem}.

\medskip

Define 
\begin{equation}\label{defn.betas}
\beta(s):=\delta_1(s)+\delta_2+\delta_4.
\end{equation}% \property{1} and \property{2}$(a).(c)$,
The next theorem discusses the error estimate for $Tf$ in $H^2$ norm that is useful to prove the error estimate of the forward direction of the problem that we are interested in, see Theorem \ref{thm.forwarderror} below.
\begin{theorem}[$H^2$ error estimate]\label{thm.uerror}
	Let $f \in H^k(\O)$ and let $T$ and $T_h$ be the solution operators introduced in \eqref{operatorT} and \eqref{operatorTh}, respectively. Then, under Assumption \ref{assumption.regularity.forward}, \property{1} and \property{2}$(a).(c)$ in Assumption \ref{assumption.forwardoperators}, there exists a constant $C>0$ independent of $f$ and $h$ such that
	\begin{equation}
	\| Tf-T_{h} f\|_{h} \le C  \beta(s)\|f\|_k,
	\end{equation}
	where $\beta(s)$ is as in \eqref{defn.betas}.
\end{theorem}
\begin{proof}
	Let $u=Tf$ and $u_h = T_h f$. The triangle inequality with  $I_h u$ and \property{1} in Assumption \ref{assumption.forwardoperators} lead to
	\begin{align}\label{h2errorIM}
	\|u-u_h\|_h &\le \|u-I_h u\|_h+\|I_h u-u_h\|_h\lesssim\delta_1(s)\|u\|_{s}+\|I_h u-u_h\|_h.
	\end{align}
Let $w_h=u_h-I_h u$. A Cauchy-Schwarz inequality and \property{1} in Assumption \ref{assumption.forwardoperators} show that
	\begin{align}\label{wh.interpolation}
		\|w_h\|_h^2&=a_h(u_h-u,w_h)+a_h(u-I_h u,w_h)\lesssim a_h(u_h-u,w_h)+ \delta_1(s)\|u\|_{s}\|w_h\|_h.
	\end{align} The definitions of $a_h(\cdot,\cdot)$, and $a(\cdot,\cdot)$, the Cauchy-Schwarz inequality and \property{2}$(a).(c)$ in Assumption \ref{assumption.forwardoperators} imply
	\begin{align*}
	a_h(u_h-u,w_h)&=(f,w_h)-a_h(u,w_h)\\
	&=(f,w_h)-a_h(u,w_h -E_h w_h)) -a_h(u,E_h w_h)\\
	&=(f,w_h-E_h w_h)-a_h(u,w_h-E_h w_h))\lesssim \big(\delta_2\|f\|+\delta_4\|u\|_{s}\big)\|w_h \|_h.
	\end{align*}
This and \eqref{wh.interpolation} prove $\|I_h u-u_h\|_h \lesssim \delta_2\|f\|+(\delta_1(s)+\delta_4)\|u\|_{s}$. The combination of this, Assumption \ref{assumption.regularity.forward} and \eqref{h2errorIM} concludes the proof.
%	$$\| Tf-T_{\rm NC} f\|_{\rm NC} \lesssim (\delta_1+\delta_2+\delta_4)\|f\|_k.$$
\end{proof}
\subsubsection{Measurement model and error estimate}\label{sec.measurement}
This section presents the measurement model and is followed by the approximation error of the forward problem. The inverse problem is to determine stable approximations for the source field $f$ of the biharmonic equation, provided the  displacement field $u$ is given only at a finite number of locations. We assume that $u$ can only be measured using a finite set of sensors, and that the output of those sensors can be modeled as continuous linear functionals on the field $u$. The measurement $m \in \R^N$ is thus modeled as
\begin{align}\label{measurement}
m=\begin{pmatrix}
	\langle \phi_1,u\rangle\\\
	 \langle \phi_2,u \rangle\\
	\cdot\\
	\cdot\\
	\cdot\\
	\langle \phi_N,u \rangle \\
\end{pmatrix}=\Phi u,
\end{align}
where the measurement functionals $u \rightarrow \langle \phi_i,u\rangle $ are assumed to be continuous in $V$, i.e, $\phi_i \in V'$ for $i=1,\cdots,N$. Here, $u$ is the solution to \eqref{weak.forward} with some source $f$. Thus, for a given source $f$, the measurement $m$ can be recast as %This allows us to model the measurement for a given source $f$ as
\begin{equation}\label{measurement.f}
m=\Phi Tf.
\end{equation}
This is now the forward direction of the problem we are particularly interested in. The corresponding finite element approximation of this forward problem is
\begin{equation}\label{measurement.cfem}
m_h=\Phi T_hf.
\end{equation}
For any $v \in V$,
$$\| \Phi v \| \le \left(\sum_{i=1}^{N}\| \phi_i \|_{-2}^2\right)^{1/2}\| v \|_2,$$
where $\|\cdot\|$ denotes the Euclidean norm on $\R^N$.  Since the functionals $\phi_i$ are fixed for a particular problem, we include their contribution as a constant $C$ and write the last inequality as
\begin{equation*}
\| \Phi v \| \le C \| v\|_2.
\end{equation*}
This, \eqref{stability} and \eqref{stability.fem} imply
\begin{equation}\label{stability.measurement}
\| \Phi Tf \| \le C\| Tf \|_2 \le C \|f \| \mbox{ and } \| \Phi T_hf \| \le C\| T_hf \|_2 \le C \|f \|.
\end{equation}
Note that the above $C$ depends on $N$, but independent of  $f$.

\begin{remark}
Even though the constant $C$ appearing in \eqref{stability.measurement} depends on $N$, we can get rid of the dependency of $N$ in the proof of the boundedness of $\Phi$ by defining a new norm on $\R^N$. For that, let $w_1,\cdots,w_N$ be such that $\sum_{i=1}^{N}w_i \le C_0$ for all $N \in \N$ and for $a=(a_i)_{i=1}^N$ and $b=(b_i)_{i=1}^N$ in $\R^N$, let $\langle a, b\rangle_w:=\sum_{i=1}^N w_ia_ib_i$. This implies $\| a\|_w^2=\sum_{i=1}^N w_ia_i^2$. Hence,
$$\| \Phi v\|_w^2=\sum_{i=1}^Nw_i |\langle \phi_i,v \rangle|^2 \le (\sum_{i=1}^Nw_i\| \phi_i \|^2)\| v\|^2.$$
Assume that for all $N \in \N$, there exists a $\beta>0$ such that $\|\phi_i\| \le \beta$ for all $i=1,\cdots,N$. Then, we have $\| \Phi v\|_w^2 \le \beta^2C_0\| v\|_2^2.$

\end{remark}

 To establish the error estimates, consider the auxiliary problem that seeks $\xi_i \in V$ such that
\begin{equation}\label{aux}
a(\xi_i,v)=(\phi_i,v)\quad \forall\, v \in V,\,\; i=1,\cdots,N.
\end{equation}
This and the definition of the solution operator given in \eqref{operatorT} lead to
$$T\phi_i=\xi_i \quad i=1,\cdots,N.$$

\begin{assumption}[Measurement function regularity]\label{assumption}
Given the measurement functionals $u \rightarrow (\phi_i,u)$, the measurement functions $\xi_i$ are in $H^{r}(\O)$, $r\ge 2$ and fixed for the particular problem. 
%Hence there exists $C>0$ such that
%	$$\| \xi_i \|_{r}\le C \quad i=1,\cdots,N.$$
\end{assumption}
Define
\begin{equation}\label{def.widetilde}
\widetilde{\beta}(r,s):=\delta_1(r)\beta(s)(1+\delta_3)+(\delta_2+\delta_5)(1+\delta_1(r)).%\delta_2+\delta_5+(1+\delta_3)\delta_1(r)\beta(s).
\end{equation}
\begin{theorem}[Approximation error of the forward problem]\label{thm.forwarderror} Let $f \in H^k(\O)$. Let $u \in H^2_0(\O)\cap H^{s}(\O)$ solves \eqref{weak.forward} and $u_h \in V_h$ be the corresponding finite element solution to \eqref{weak.forward.fem}. Under Assumptions \ref{assumption.regularity.forward}, \ref{assumption.forwardoperators} and \ref{assumption}, there exists a constant $C(N)>0$ independent of $h$ and $f$ such that 
	$$\| \Phi Tf -\Phi T_h f\| \le C(N)\widetilde{\beta}(r,s)\| f\|_{k}.$$
\end{theorem}
\begin{proof}
	Let $u=Tf$ and $u_h=T_h f$. For $i=1,\cdots,N$, \eqref{aux}, a Cauchy Schwarz inequality and \property{2}.$(a)$ in Assumption \ref{assumption.forwardoperators}, \eqref{stability.fem} and \eqref{stability.measurement} show that
\begin{align}\label{phiu-uh}
(\phi_i,u-u_h)&=(\phi_i,u-E_h u_h)+(\phi_i,E_h u_h-u_h)\nonumber\\
&\lesssim a(\xi_i,u-E_h u_h)+\delta_2\|\phi_i\|\|u_h\|_h\lesssim a(\xi_i,u-E_h u_h)+\delta_2\|f\|.
\end{align}	
The first term in the right hand side of \eqref{phiu-uh} can be rewritten as
\begin{align}\label{forwarderr}
a(\xi_i,u-E_h u_h)&=a_h(\xi_i,u-u_h)+a_h(\xi_i-I_h \xi_i, u_h-E_h u_h)+a_h(I_h \xi_i, u_h-E_h u_h)\nonumber \\
&=:T_1+T_2+T_3.
\end{align}
Cauchy-Schwarz inequalities, \property{1} and \property{2}$(b)$ in Assumption \ref{assumption.forwardoperators}, and  Theorem \ref{thm.uerror} prove
\begin{align}
T_1&= a_h(\xi_i-I_h \xi_i,u-u_h)+a_h(I_h \xi_i-E_h I_h \xi_i,u- u_h)+a_h(E_h I_h \xi_i,u- u_h)\nonumber\\
&\lesssim (\|\xi_i-I_h \xi_i\|_h +\|I_h\xi_i-E_hI_h \xi_i\|_h) \|u-u_h\|_h+a_h(E_hI_h \xi_i,u- u_h)\nonumber \\
&\lesssim \delta_1(r)\beta(s)(1+\delta_3)\|f\|_k\|\xi_i\|_r+a_h(E_h I_h \xi_i,u- u_h)\label{t1.forward}
\end{align}
with $\|I_h\xi_i-E_hI_h \xi_i\|_h\le \delta_3\|I_h\xi_i-\xi_i\|_h\le \delta_1(r)\delta_3\|\xi_i\|_r$ in the end. The triangle inequality with $\xi_i$ and \property{1} in Assumption \ref{assumption.forwardoperators} verify $\|I_h\xi_i\|_h \lesssim (1+\delta_1(r))\|\xi_i\|_r$. This, the definitions of $a(\cdot,\cdot)$ and $a_h(\cdot,\cdot)$, \property{1} and \property{2}$(a).(d)$ in Assumption \ref{assumption.forwardoperators}, and \eqref{stability.fem} lead to
\begin{align*}
a_h(E_h I_h \xi_i,u- u_h)
&=a( E_h I_h \xi_i,u)-a_h(E_h I_h \xi_i -I_h \xi_i, u_h)-a_h(I_h \xi_i, u_h)\\
&=(f,E_h I_h \xi_i-I_h \xi_i)-a_h(E_h I_h \xi_i -I_h \xi_i, u_h)\nonumber\\
&\lesssim \delta_1(r)\delta_2\|f\|\|\xi_i\|_{r}+\delta_5\| I_h\xi\|_h\|u_h\|_h\nonumber\\
& \lesssim (\delta_1(r)\delta_2+\delta_5(1+\delta_1(r)))\|f\|\|\xi_i\|_{r}.
\end{align*}
Consequently, \eqref{t1.forward} implies
\begin{align}
T_1&\lesssim \big(\delta_1(r)\beta(s)(1+\delta_3)+\delta_1(r)\delta_2+\delta_5(1+\delta_1(r))\big)\|f\|_k\|\xi_i\|_r.\label{t1.forward1}
\end{align}
A Cauchy Schwarz inequality, \property{1} and \property{2}$(b)$ in Assumption \ref{assumption.forwardoperators}, and Theorem \ref{thm.uerror} prove
\begin{align}
T_2&\lesssim \|\xi_i-I_h \xi_i \|_h\|u_h-E_h u_h \|_h\lesssim \delta_1(r)\delta_3\|\xi_i\|_r\|u_h-u \|_h  \lesssim \delta_1(r)\beta(s)\delta_3\|f\|_k\|\xi_i\|_r.\label{t2.forward}
\end{align}
Since $\|I_h\xi_i\|_h \lesssim (1+\delta_1(r))\|\xi_i\|_r$, \property{2}$(d)$ in Assumption \ref{assumption.forwardoperators} shows
\begin{equation}\label{t3.forward}
T_3 \lesssim \delta_5\|I_h\xi_i\|_h\|u_h\|_h \lesssim \delta_5(1+\delta_1(r))\|\xi_i\|_r\|f\|_k
\end{equation}
with \eqref{stability.fem} in the last step. The combination of \eqref{t1.forward1}-\eqref{t3.forward} in \eqref{forwarderr} reads
$$a(\xi_i,u-E_h u_h)\lesssim \big(\delta_1(r)(\delta_2+\beta(s)(1+\delta_3))+\delta_5(1+\delta_1(r))\big)\|f\|_k\|\xi_i\|_{r}.$$ This with \eqref{phiu-uh} result in
$$(\phi_i, u-u_h) \lesssim \big(\delta_1(r)\beta(s)(1+\delta_3)+(\delta_2+\delta_5)(1+\delta_1(r))\big)\|f\|_k.$$
Consequently, the desired estimate follows from the definition of $\Phi$.
\end{proof}

\subsection{Finite Element Methods} \label{sec:fem}
The applications of the results in Section \ref{sec.forwardproblem} to various schemes are discussed in this section.

\smallskip

Let $\cT_h$ be a regular and conforming triangulation of $\overline{\Omega}$ into closed triangles, rectangles or quadrilaterals. Set $h_K=$ diameter of $K \; \;  \forall\, K \in \cT_h $ and define the discretization parameter $h:=\max_{K\in\cT_h}h_K$. 

\subsubsection{Conforming FEMs} \label{sec.cfem.forward}
Two examples of conforming finite elements, namely the Argyris triangle and \BFS rectangle, are presented below. For conforming finite elements, the finite element space $V_\C$ is a subspace of the underlying Hilbert space $V$.

\medskip

$\bullet$ \textsc{The Argyris triangle  \cite{ciarlet1978finite}:}
The Argyris triangle  (Figure ~\ref{fig.cfem}, left) is a triplet $(K,\cP_K, \Sigma_K)$ where $K$ is a triangle with vertices $a_1,\,a_2,\,a_3$ and $a_{ij}=\frac{1}{2}(a_i+a_j),\, 1\le i <j \le 3$ denote the midpoints of the edges of $K$, $\cP_K$ = $\cP_5(K)$, space of all polynomials of degree $\le$ 5 in two variables defined on
$K$ (dim $\cP_K=21$), and $\Sigma_K$ denotes the degrees of freedom given by: for $p \in \cP_K$,
\[
\begin{aligned}
\Sigma_K=\bigg\{p(a_i),\, \partial_1 p(a_i),\,& \partial_2 p(a_i),\,\partial_{11} p(a_i), \,\partial_{12}p(a_i),\, \partial_{22} p(a_i),\, 1\le i \le 3;\\
& \frac{\partial p}{\partial n}(a_{ij}), \,1\le i <j \le 3\bigg\}.
\end{aligned}
\]

\begin{figure}
	\begin{center}
		\begin{minipage}[H]{0.25\linewidth}
			{\includegraphics[width=3.3cm]{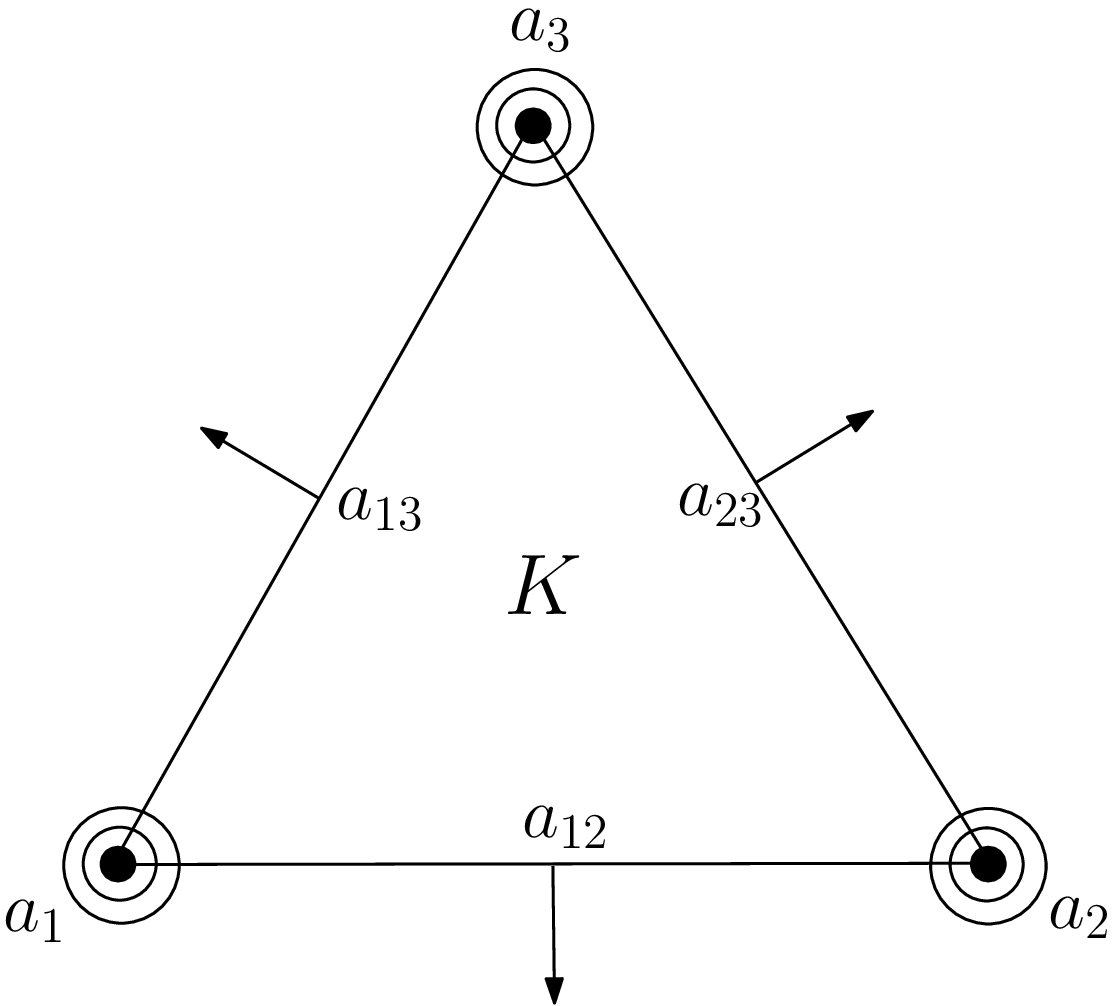}}
			%	\caption{Argyris triangle}\label{fig.argyris}
		\end{minipage}
		\begin{minipage}[H]{0.25\linewidth}
			{\includegraphics[width=3.7cm]{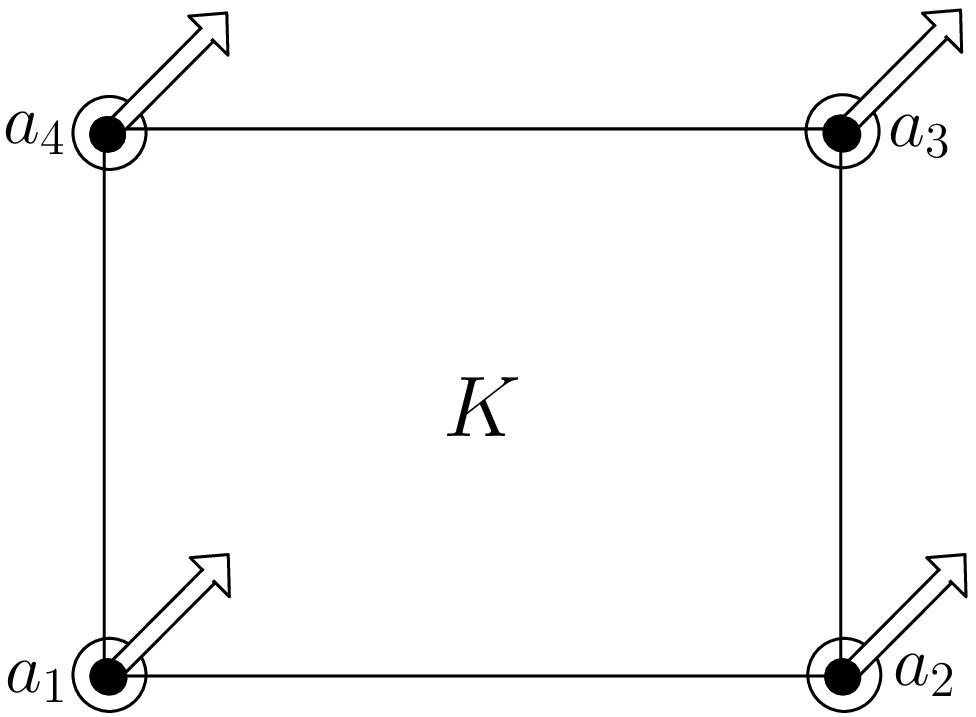}}
			%	\caption{Bogner-Fox-Schmit rectangle }\label{fig.bfs}
		\end{minipage}
		\begin{minipage}[H]{0.25\linewidth}
		{\includegraphics[width=3.3cm]{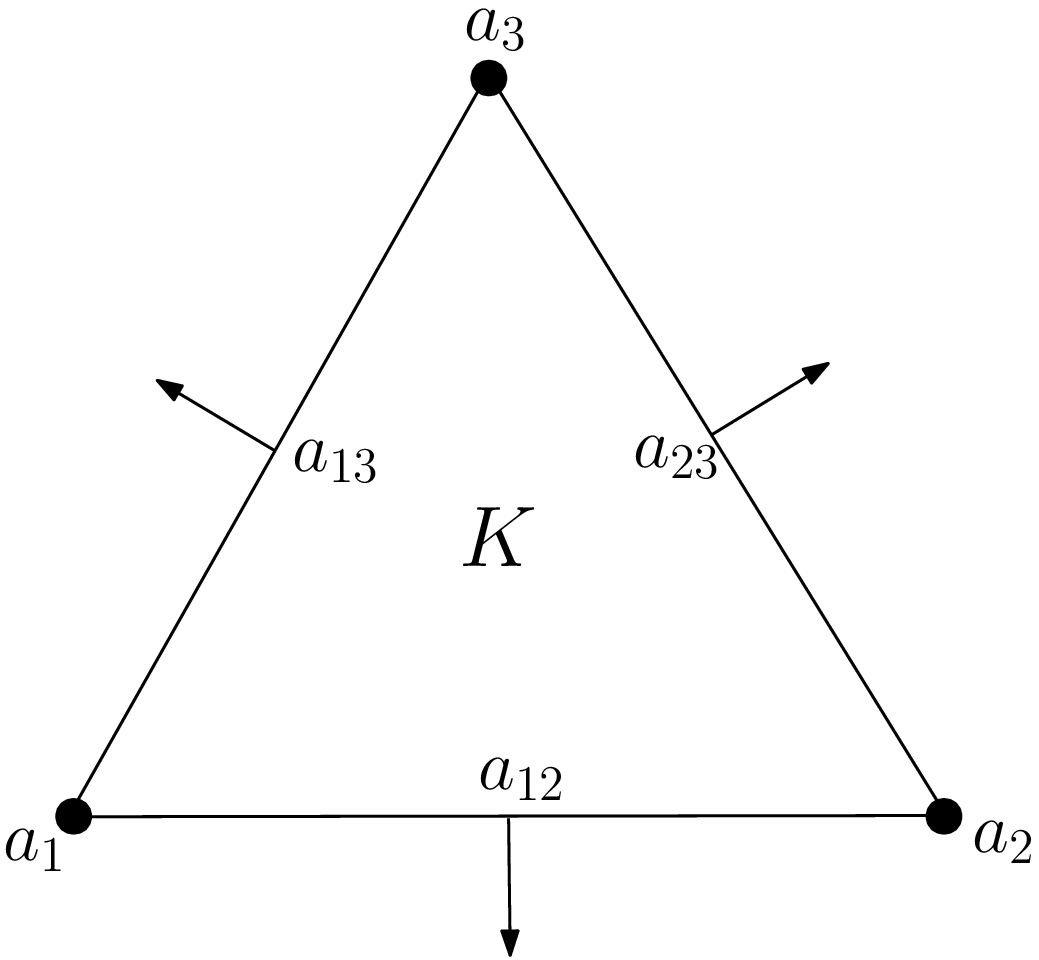}}
	%	\caption{Morley triangle}\label{ncfem.fig}
	\end{minipage}
		\caption{Argyris triangle (left), Bogner-Fox-Schmit rectangle (center) and Morley triangle (right) }\label{fig.cfem}
	\end{center}
\end{figure}

$\bullet$ \textsc{Bogner-Fox-Schmit rectangle \cite{ciarlet1978finite}: }The Bogner-Fox-Schmit rectangle (Figure ~\ref{fig.cfem}, center) is a triplet $(K,\cP_K, \Sigma_K)$ where $K$ is a rectangle with vertices $a_i, \, 1 \le i \le 4$, $\cP_K$ = $Q_3(K)$, the polynomials of degree $\le$ 3 in both variables (dim $\cP_K=16$), and $\Sigma_K$ is given by:
%	\begin{align*}
$$	\Sigma_K=\big\{p(a_i),\,\partial_1 p(a_i), \,\partial_2 p(a_i),\,\partial_{12}p(a_i),\, 1\le i \le 4 \big\}.$$

The conforming finite element spaces associated with \BFS  and Argyris elements are contained in
$ C^1(\overline\Omega)\cap H^2(\Omega)$. Define
\begin{align*}
V_\C&=\left\{v_\C\in C^1(\overline\Omega):v_\C|_{K}\in \cP_K \quad \forall\, K\in\cT_h \text{ with }v_\C|_{\partial\Omega}=0, \frac{\partial v_\C}{\partial n}\big{|}_{\partial\Omega}=0\right\}\subset  V,
% X_{\C}&=\{v\in X_{\C}: v_{|\partial\Omega}=0, \frac{\partial v}{\partial \nu}\big{|}_{\partial\Omega}=0 \}\subset \hto,
\end{align*}
where
\begin{align*} 
\cP_K &= \left\{ 
\begin{array}{l l}
\cQ_3(K)& \; \text{for \BFS element}, \\
\cP_5(K) & \; \text{for Argyris element.}
\end{array} \right.
\end{align*}
%Let $I_h:V\map V_{h}$ be the \BFS or Argyris nodal interpolation operator. 
The following lemma talks about the interpolation estimates for the \BFS or Argyris FEMs.

\begin{lemma}[Interpolant]\cite{ciarlet1978finite}\label{interpolant.cfem} For $v\in V \cap H^{s}(\Omega)$,
	\begin{align*}
	&\|v-I_\C v\|_{\ell} \leq C h^{\min\{p+1,\;s\}-\ell}\|v\|_{s} \qquad\text{ for } \ell=0,1,2,
	\end{align*}
	where $p=3$ (resp. $5$) for the \BFS element  (resp. Argyris element), and $C>0$ is independent of $h$ and $v$.
\end{lemma} 

{\bf Discussion of \property{1}-\property{2} in Assumption \ref{assumption.forwardoperators}:} For conforming FEMs, 
$$V_h:=V_\C, \;a_h(\cdot,\cdot):=a(\cdot,\cdot),\;\|\cdot\|_h:=|\cdot|_2,\;I_h:=I_\C,\;\mbox{ and } E_h:=E_\C.$$
The interpolation operator in Lemma \ref{interpolant.cfem} satisfies the condition in \property{1} with $\delta_1(s)\lesssim h^{\min\{p+1,s\}-2}$ where $p=3$ (resp. $5$) for the \BFS element  (resp. Argyris element). Since the operator $E_h$ for the conforming FEMs is the identity operator $\rm{Id}$ as $V_\C \subset V$, the conditions $(a)-(d)$ in \property{2} are trivially satisfied with $\delta_2=\delta_3=\delta_4=\delta_5=0$. As a result,
Theorems \ref{thm.uerror} and \ref{thm.forwarderror} lead to
$$\|Tf-T_hf\|_2\le Ch^{\min\{p+1,s\}-2}\|f\|_k $$
and
\begin{equation}\label{thm.forwarderror.cfem}
\| \Phi Tf -\Phi T_h f\| \le C(N)h^{\min\{p+1,\;s\}+\min\{p+1,\;r\}-4}\| f\|_{k}.
\end{equation}
\subsubsection{Nonconforming FEM}\label{sec.ncfem.forward}
The well-known nonconforming finite element \cite{ciarlet1978finite}, the Morley element, is discussed below.

\medskip

%$\bullet$ \textsc{the Morley element: }
 For a triangle $K \in \cT_h$ with vertices $a_1,\,a_2,\,a_3$, let $a_{12},\, a_{23}$ and $a_{13}$ denote the midpoint of the edges opposite to the vertices $a_3,\, a_1$ and $a_2$, respectively (Figure ~\ref{fig.cfem}, right). The Morley finite element is a triplet $(K,\cP_K, \Sigma_K)$ where $K$ is a triangle in $\mesh$, $\cP_K$ = $\Poly_2(K)$, space of all polynomials of degree $\le$ 2 in two variables defined on $K$ (dim $\cP_K=6$) and $\Sigma_K$ denotes the degrees of freedom given by: $$\Sigma_K=\bigg\{p(a_i), 1\le i \le 3; \frac{\partial p}{\partial n}(a_{ij}), 1\le i <j \le 3\bigg\}.$$
The nonconforming {\it Morley element space} $V_\M$  is defined by
\[
V_{\M}:=\left \{ v_\M \in {\mathcal P}_2(\cT_h){{\Bigg |}}
\begin{aligned}
& v_\M \text{ is continuous at the interior vertices}  \text{ and vanishes at the } \\
& \text{vertices of }\partial \Omega; 
D_{\NC}{ v_\M} \text{ is continuous at the midpoints of interior}
\\
& \text{edges and vanishes at the  midpoints of the edges of } \partial \Omega
\end{aligned}
\right\}
\]
and  is equipped with the norm $\|\cdot\|_{\NC}$ defined by  $ \|v_\M\|_{\NC}= \bigg(\sum_{K\in \mathcal{T}_h}\|D^2_\NC v_\M\|_{L^2(K)}^2\bigg)^{1/2}$.  Here $D_{\NC} \cdot$ and  $D^2_\NC \cdot$ denote the piecewise gradient and second derivative of the arguments on triangles  $K \in \mathcal{T}_h$. For $v \in V$, we have $\|v\|_2=\|v\|_\NC$ and thus $\| \cdot \|_\NC$ denotes the norm in $V+V_\M$. The discrete bilinear form is defined by $$a_h(\eta_\M,\chi_\M)=:a_\NC(\eta_\M,\chi_\M)=\sit D^2 \eta_\M:D^2\chi_\M \dx \quad \forall\, \eta_\M,\chi_\M \in V_ {\M},
$$ where $:$ denotes the scalar product of matrices.

\smallskip

Some auxiliary results needed in the discussion are stated below.

\begin{lemma}[\it Morley interpolation operator] \label{Morley_Interpolation} 
	\cite{CCDGJH14, DG_Morley_Eigen, CCP} 
	The Morley interpolation operator 
	$I_{\rm M}: V \rightarrow V_{\rm M}$ defined by
	\begin{align*}
	& (I_{\rm M} v)(z)=v(z) \text{ for any } \text{ vertex $z$ of } \cT_h \text{ and } 
	v\in V, \\
	&
	\int_E\frac{\partial I_{\rm M} v}{\partial n_E}\ds=\int_E\frac{\partial v}{\partial n_E}\ds \text{ for any  edge } E  \text{ of } \cT_h \text{ and } 
	v \in V,
	\end{align*}
	satisfies the integral mean property $D^2_{\rm NC} I_{\rm M} =\Pi_0 D_{\rm NC}^2$ (where $\Pi_0$ denotes the projection onto the space of piecewise constants ${\mathcal P}_0(\cT_h)$) of the Hessian.  For $K \in \cT_h$ with diameter $h_K$, we have the approximation and stability properties as%stated below hold. 
	%The function $v\in V+ \cM(\cT)$ and its interpolation $I_M v$ satisfy
	%\begin{equation*}
\[	\|h_K^{-2}(1-I_{\rm M})v\|_{L^2(K)}+\|h_K^{-1}D_{\rm NC}(1-I_{\rm M})v\|_{L^2(K)} \lesssim
	\|D_{\rm NC}^2 (1-I_{\rm M})v\|_{L^2(K)}. \]
	%\end{equation*}
	Moreover, for $v \in V \cap H^{2+\gamma}(\Omega)$, 
	\[\|D_{\rm NC}^2 (1-I_{\rm M})v\|\lesssim h^{\gamma} \| v\|_{{2+\gamma}},\] where $\gamma \in (1/2,1]$ is the index of elliptic regularity. 
	Also, if $v\in V \cap H^{s}(\Omega)$, then
	\begin{align*}
	&\|v-I_{\rm M}v\|_{h} \lesssim h^{\min\{3,\;s\}-2}\|v\|_{s}.
	\end{align*}
\end{lemma}

%{\color{blue} Companion Operator}

\begin{lemma}[\it Enrichment operator]\label{hctenrich} \cite{DG_Morley_Eigen,SCNNDS} For any $v_{\rm M} \in 
	V_{{\rm M}}$,  there exists an enrichment operator $E_{\rm M}:{V}_{{\rm M}}\to V$ such that for all $v_{\rm M} \in V_{{\rm M}}$,
	\begin{align*}
	&(a)\, I_{\rm M} E_{\rm M} v_{\rm M}= v_{\rm M}, \;
	(b)\,\Pi_0(v_{\rm M} - E_{\rm M}  v_{\rm M}) =0, \;
	\; (c)\,\Pi_0D_{\rm NC}^2(v_{\rm M} - E_{\rm M}  v_{\rm M}) =0, \\
	& (d)\,\| h_{K}^{-2}(v_{\rm M}-E_{\rm M}v_{\rm M})\| + \| h_{K}^{-1}D_{\rm NC}(v_{\rm M}-E_{\rm M}v_{\rm M})\| + \| D^2_{\rm NC}(v_{\rm M}-E_{\rm M}v_{\rm M})\| \lesssim \min_{v \in V}\| D_{\rm NC}^2(v_{\rm M}-v)\|. 
	\end{align*}
\end{lemma}
\begin{lemma}\cite[ Lemma 4.2]{ScbSungZhang}\label{Anc.bound} 
	%	\begin{itemize}
	%	\item[(i)]
	For $\chi \in H^{2+\gamma}(\Omega)$ and $\phi_{\rm M} \in V_{\rm M}$, it  holds that 
	$$a_{\rm NC}(\chi,E_{\rm M}\phi_{\rm M}-\phi_{\rm M})\lesssim h^\gamma\| \chi \|_{2+\gamma}\| \phi_{\rm M} \|_{\rm NC}.$$
	%		\item[(ii)]Further, if $\chi \in H^{2+\gamma}(\Omega)$ and $\phi \in V\cap H^{2+\gamma}(\Omega)$, then 
	%		$$a_{\rm NC}(\chi,I_{\rm M} \phi-\phi)\lesssim h^{2\gamma}\| \chi \|_{2+\gamma}\| \phi \|_{2+\gamma}.$$
	%	\end{itemize}
\end{lemma}
The above result can be proved using an integration by parts, and the enrichment estimate in energy and piecewise $H^1$ norms in Lemma \ref{hctenrich} $(d)$. %If we assume more smoothness of the function $\chi$, then a better estimate can be derived by using again an integration by parts and the properties of $E_\M$ in Lemma \ref{hctenrich}. 
Hence, we arrive at the following lemma.

\begin{lemma}[\it Bound for $a_{\rm NC}(\cdot,\cdot)$]\label{anc.bound}\noindent
	For $\chi \in H^{s}(\Omega)$ and $\phi_{\rm M} \in V_{\rm M}$, it  holds that 
	$$a_{\rm NC}(\chi,E_{\rm M}\phi_{\rm M}-\phi_{\rm M})\lesssim h^{\min\{1,s-2\}}\| \chi \|_{s}\| \phi_{\rm M} \|_{\rm NC}.$$
\end{lemma}

{\bf Discussion of \property{1}-\property{2} in Assumption \ref{assumption.forwardoperators}:} In this case, 
$$V_h:=V_\M,\;a_h(\cdot,\cdot):=a_\NC(\cdot,\cdot),\;\|\cdot\|_h:=\|\cdot\|_\NC,\;I_h:=I_\M,\;\mbox{ and } E_h:=E_\M.$$ 
The definition of $a_\NC(\cdot,\cdot)$ and the norm $\|\cdot\|_\NC$ verify the coercivity and boundedness properties of $a_\NC(\cdot,\cdot)$. From Lemma \ref{Morley_Interpolation} and \property{1}, we have $\delta_1(s)\lesssim h^{\min\{1,s-2\}}$. Lemma \ref{hctenrich}$(d)$ and \property{2}$(a)-(b)$ show that $\delta_2\lesssim h^2, \delta_3\lesssim 1$. Lemma \ref{anc.bound} and \property{2}$(c)$ imply $\delta_4\lesssim h^{\min\{1,s-2\}}$.
Since the piecewise second derivatives of $v_\M \in V_\M$ are constants, the orthogonality property in Lemma \ref{hctenrich}$(c)$ and \property{2}$(d)$ yield $\delta_5=0$. Consequently,
$$\|Tf-T_hf\|_h\le Ch^{\min\{1,s-2\}}\|f\|_k $$
and
\begin{equation}\label{thm.forwarderror.ncfem}
\| \Phi Tf -\Phi T_h f\| \le C(N)h^{\min\{1,s-2\}+\min\{1,r-2\}}\| f\|_{k}.
\end{equation}
As seen above, the order of convergence of the error estimate in piecewise $H^2$ norm for Morley FEM is atmost 1 even if $u \in H^s(\O)$ for $s \ge 3$. This is due to the fact that $V_\M$ is a low-order scheme.

\begin{remark}
	It is interesting to note that the first term in the right hand side of \eqref{wh.interpolation} for conforming FEM is 0, that is, $a(u_\C-u,w_\C)=0$ for all $w_\C \in V_\C$ since $V_\C \subset V$. Also, since the Morley interpolation operator $I_\M$ in Lemma \ref{Morley_Interpolation} satisfies $D^2_{\rm NC} I_{\rm M} =\Pi_0 D_{\rm NC}^2$ and the second derivatives of $w_\M \in V_\M$ are constants, the second term in the right hand side of \eqref{wh.interpolation} for Morley nonconforming FEM is 0.
	
	\smallskip
	
	For $f \in L^2(\O)$ with $u,\xi \in H^{2+\gamma}(\O)$, results for Argyris, \BFS conforming FEMs and Morley FEM yield
	$$\|Tf-T_hf\|_2\le Ch^{\gamma}\|f\| ,\quad \| \Phi Tf -\Phi T_h f\| \le C(N)h^{2\gamma}\| f\|.$$
\end{remark}

\section{Inverse problem}\label{sec.inverseproblem}
The forward problem is analysed in Section \ref{sec.forwardproblem} whereas this section deals with the inverse problem which can be referred to as the inverse of the forward problem. Given the measurement $m$ of the displacement field $u$, the inverse problem is to determine the source field $f$ of the biharmonic equation.  We will find approximations of the solution of the inverse problem by taking a finite number of displacement readings. The inverse problem is ill-posed as the solution does not depend continuously on the measurement data as the operator $T$ is a compact operator of infinite rank. In order to obtain a stable approximate solution, we consider reconstructions obtained using the Tikhonov regularization method \cite{IP_HEMHAN, MTN_book_LOE}.
\medskip

%\cblue{Suppose $u \in R(T)$ and $f \in N(T)^\perp$ is the unique element such that $Tf=u$. For $\alpha>0$, let $f_\alpha^\alpha \in L^2(\O)$ be the unique element such that $(T^*T+\alpha I)f_\alpha^\alpha=T^*u$. Then it is well-known that $f_\alpha^\alpha \rightarrow f$ as $\alpha \rightarrow 0$.} 
Let $F$ be a closed subspace of $H^k(\O)$ where the source field reconstruction exists, $k \in \N_0$. We also consider a symmetric bilinear form $b : F \times F \rightarrow \R$, which is continuous (resp. coercive) in $F$ with continuity (resp. coercivity) constant $C_1$ (resp. $C_2$) with respect to the norm $\|\cdot\|_k$. For instance, $F$ can be taken as $L^2(\O)$ (resp. $H^1_0(\O)$) for $k=0$ with $b(f,g)=(f,g)$ (resp. $k=1$ with $b(f,g)=\int_\O \nabla f\cdot \nabla g \dx$).
%\begin{itemize}
%	\item continuous in $F$: there exists a constant $C_1>0$ such that
%	$$b(f,g)\le C_1 \| f\|_{k}\| g \|_{k} \quad \forall\, f,g \in F.$$
%	\item coercive in $F$: there exists a constant $C_2>0$ such that
%	$$b(f,f)\ge C_2\| f\|_{k}^2.$$
%\end{itemize}
\smallskip

We may observe that for every  $f \in F$, 
\begin{equation}\label{h-2tol2}
\| f\|_{-2} \le \|  f \|_k.
\end{equation}
Let the actual source field be denoted by $f_{\t}$ and assume that $f_{\t} \in F$. The reconstructed regularised approximation of the source field $f_\alpha$ is then obtained as a Tikhonov regularized solution of the inverse problem, defined by the minimization problem
$$f_\alpha=\argmin_{f \in F}\big\{\|m-\Phi Tf\|^2+\alpha b(f,f)\big\},$$
where $\alpha>0$ is the regularization parameter to be chosen appropriately, and $m=\Phi Tf_{\t}$. Define
$$J(f):=\|m-\Phi Tf\|^2+\alpha b(f,f).$$
Thus, $J(f_\a)=\min_{f \in F}J(f)$. A necessary condition for $f_\a$ to be a minimiser of $J(\cdot)$ is that the derivative of $J(\cdot)$ vanishes at $f_\a$, so that
\begin{equation*}%\label{min.derivative}
\lim_{\beta \to 0}\frac{J(f_\a + \beta f)-J(f_\a)}{\beta}=0 \fl f \in F.
\end{equation*}
Consider ${J(f_\a + \beta f)-J(f_\a)}$. Using the definition of $J(\cdot)$ and the bilinear and symmetric properties of $b(\cdot,\cdot)$, we obtain
\begin{align*}
{J(f_\a + \beta f)-J(f_\a)}&
%=\|m-\Phi T(f_\a +\beta f)\|^2+\alpha b(f_\a +\beta f,f_\a +\beta f)\\&\quad -\|m-\Phi Tf_\a\|^2-\alpha b(f_\a,f_\a)\\
%&=(m-\Phi T(f_\a +\beta f))^*(m-\Phi T(f_\a +\beta f))\nonumber\\
%&\quad +\alpha b(f_\a +\beta f,f_\a +\beta f) -(m-\Phi Tf_\a)^*(m-\Phi Tf_\a)\\&\quad -\alpha b(f_\a,f_\a)\\
&=2\beta^2\Phi Tf-2\beta(m-\Phi Tf_\a)^*(\Phi Tf)+2\beta\alpha b(f_\a,f)+\beta^2 \alpha b(f,f).
\end{align*}
Dividing the above equation by $\beta$ and taking the limit $\beta \to 0$, we deduce
$$(\Phi Tf_\a)^*(\Phi Tf)+\alpha b(f_\a,f) -m^*(\Phi Tf)=0.$$

\medskip

Therefore, the minimisation problem can be written equivalently in variational form: seeks $f_\alpha \in F$ such that
\begin{equation}\label{weak.inverse}
B(f_\alpha,g)=l(g) \quad \forall\, g \in F
\end{equation}
where
\begin{equation}\label{weak.inverse.Bandl}
B(f,g)=(\Phi Tf)^*(\Phi Tg)+\alpha b(f,g) \mbox{ and } l(g)=m^*(\Phi Tg).
\end{equation}

%where $C_p$ is the Poincar\'e constant of $\O$.
The proof of the next lemma follows from a Cauchy-Schwarz inequality, the continuity and coercivity properties of $b(\cdot,\cdot)$, \eqref{stability.measurement} and \eqref{h-2tol2} and hence is skipped.
\begin{lemma}\label{lem.bilinearB}
	The bilinear form $B(\cdot,\cdot)$ on $F$ is continuous and coercive.
\end{lemma}
%\begin{proof}
%$(i)$ Continuity: Let $f, g \in F$. Use the Cauchy-Schwarz inequality, the continuity of $b(\cdot,\cdot)$, \eqref{stability.measurement} and \eqref{h-2tol2} to obtain
%	\begin{align*}
%	|B(f,g)|\le\|\Phi Tf\|\|\Phi Tg\|+\alpha C_1 \|f\|_k\|g\|_k&\le (C(N)^2+\alpha C_1)\|f\|_k\|g\|_k\\
%	&= C \|f\|_k\|g\|_k,
%	\end{align*}
%	where $C$ depends on $\alpha$ and $N$, independent of $f$ and $g$.
%	
%	\smallskip
%	
%$(ii)$ Coercivity: The coercivity of of $b(\cdot,\cdot)$ leads to
%\begin{align*}
%B(f,f)=(\Phi Tf)^*(\Phi Tf)+\alpha b(f,f)=\| \Phi Tf \|^2 +\alpha b(f,f) \ge \alpha b(f,f)\ge \alpha C_2 \|f\|_k^2,
%\end{align*}
%where $C_2>0$ is independent of $\alpha$ and $f$.	
%	\end{proof}
Also, the linear functional $l(\cdot)$ in \eqref{weak.inverse} is continuous in the Hilbert space $F$. Hence, by the Lax-Milgram lemma, there exists a unique solution $f_\alpha$ to \eqref{weak.inverse} and it holds that
%$$\alpha C_2\|f_\alpha\|_k^2\le B(f_\alpha,f_\alpha)=l(f_\alpha)=m^*(\Phi Tf_\alpha)\le \|m\| \|\Phi Tf_\alpha\|\le C(N)\|m\|\|f_\alpha\|_k.$$ 
\begin{equation}\label{stability.fr}
\|f_\alpha\|_k \le \frac{C(N)}{{\alpha C_2}}\|m \|,
\end{equation}
where $C(N)>0$ is a constant depending on $N$, but independent of $\alpha$ and $m$.

\smallskip

However, it is possible to get a better bound for $f_\alpha$. For that, we consider the operator approach of the variational form \eqref{weak.inverse}.

{{Define $\mathcal{A} \in \mathcal{L}(F,\R^N)$ and $\mathcal{B} \in \mathcal{L}(F,F)$ by
		$$(\mathcal{A^*A}f)(g):=(\Phi Tf)^*(\Phi Tg) \mbox{ and } (\mathcal{B}f)(g):=b(f,g).$$
		Then, \eqref{weak.inverse} can be rewritten as:% seeks $f_\alpha \in F$ such that
		\begin{equation}\label{operator}
		(\mathcal{A^*A}+\alpha \mathcal{B})f_\a=\mathcal{A^*}m.
		\end{equation}
	%	Note that
	%	$$(\mathcal{A^*A}f)(f)=(\Phi Tf)^*(\Phi Tf)=\|\Phi Tf\|^2 \ge 0 \mbox{ and }  (\mathcal{B}f)(f)=b(f,f) \ge C_2\|f\|_k^2 \ge 0$$
%		where we used the fact that $b(\cdot,\cdot)$ is coercive in $F$. Also, $\Phi Tf, \Phi Tg \in \R^\N$ and the symmetric property of $b(\cdot,\cdot)$ shows that
%		$$(\mathcal{A^*A}f)(g)=(\Phi Tf)^*(\Phi Tg)=(\Phi Tg)^*(\Phi Tf)=(\mathcal{A^*A}g)(f)$$
%		and 
%		$$ (\mathcal{B}f)(g)=b(f,g)=b(g,f)= (\mathcal{B}g)(f).$$
%		Hence
Note that $\mathcal{A^*A}$ and $\mathcal{B}$ are positive self-adjoint operators in $F$. Since $\mathcal{A^*A}$ is positive, the coercivity property of $b(\cdot,\cdot)$ leads to
		\begin{align*}
		\|(\mathcal{A^*A}+\alpha \mathcal{B})f \|_k \|f \|_k& \ge ( (\mathcal{A^*A}+\alpha \mathcal{B})f,f)
		=( \mathcal{A^*A}f,f)+(\alpha \mathcal{B}f,f)\ge \alpha C_2 \| f\|_k^2\end{align*}
		for every $f \in F$. Thus,
		$$\|(\mathcal{A^*A}+\alpha \mathcal{B})f \|_k \ge \alpha C_2 \| f\|_k \quad \forall\, f \in F.$$
		Hence, $\mathcal{A^*A}+\alpha \mathcal{B}$ is injective with its range closed, and its inverse from the range is continuous. Since $\mathcal{A^*A}+\alpha \mathcal{B}$ is also self-adjoint, we have 
		$$R(\mathcal{A^*A}+\alpha \mathcal{B})=N(\mathcal{A^*A}+\alpha \mathcal{B})^\perp=F.$$
		Thus, $\mathcal{A^*A}+\alpha \mathcal{B}$ is onto as well and hence it is invertible. Consequently, \eqref{operator} results in
		\begin{equation}\label{operator.falpha}
f_\a=(\mathcal{A^*A}+\alpha \mathcal{B})^{-1}\mathcal{A^*A}f_{\t}.
		\end{equation}
		Observe that
		$(\mathcal{A^*A}+\alpha \mathcal{B})^{-1}\mathcal{A^*A}=I-\alpha\mathcal{B}(\mathcal{A^*A}+\alpha \mathcal{B})^{-1}.$ Since $(\alpha\mathcal{B}(\mathcal{A^*A}+\alpha \mathcal{B})^{-1}f,f)\ge 0$ for every $f \in F$, 
		\begin{align*}
		( (\mathcal{A^*A}+\alpha \mathcal{B})^{-1}\mathcal{A^*A}f,f )&=( f,f)-(\alpha\mathcal{B}(\mathcal{A^*A}+\alpha \mathcal{B})^{-1}f,f)\le ( f,f).
		\end{align*}
		Since $(\mathcal{A^*A}+\alpha \mathcal{B})^{-1}\mathcal{A^*A}$ is self-adjoint, the definition of operator norm and \eqref{operator.falpha} show
		\begin{equation}\label{bound.falpha}
\| f_\a \|_k =\| (\mathcal{A^*A}+\alpha \mathcal{B})^{-1}\mathcal{A^*A}f_{\t}\|_k\le \|f_{\t}\|_k.
		\end{equation}}}
%Taking the regularization as fixed for a particular problem, we include the parameter $\alpha$ in the constant C, and the coercivity can be denoted by
%$$C\|f\|^2 \le B(f,f).$$
%As a consequence, $\|f_\alpha\| \le C\|m \|$.

\subsection{Error between the actual and reconstructed sources}
This section deals with the error estimate between the actual source field $f_{\t}$ and the reconstructed source field $f_\alpha$. More precisely, we will show that $f_\a$ lies in a finite dimensional space (see \eqref{def.falpha.basis} below) and then derive the error between $f_\a$ and the best approximation obtainable for $f_{\t}$ in this space (since $f_{\t}$ does not necessarily belong to this space).

\smallskip

Define $d=\frac{1}{\alpha}(m-\Phi Tf_\a)$. Then, \eqref{weak.inverse} can be rewritten as: seeks $f_\a \in F$ such that
$$b(f_\a,g)=d^*(\Phi Tg) \quad \forall\, g \in F.$$ 
The definition of the measurement operator given by \eqref{measurement}, \eqref{aux} and the forward problem \eqref{weak.forwardoperator} imply
\begin{equation}\label{error.b}
b(f_\a,g)=\sum_{i=1}^{N}d_i( \phi_i, Tg)=\sum_{i=1}^{N}d_ia(\xi_i,Tg)=\sum_{i=1}^{N}d_i(\xi_i,g) \quad \forall\, g \in F.
\end{equation}
Let $i\in \{1,\ldots,N\}$. Define $\eta_i \in F$ as the solution to the problem given by: seeks $\eta_i \in F$ such that
\begin{equation}\label{basis_eta}
b(\eta_i,g)=(\xi_i,g) \quad \forall\, g \in F.\end{equation}
This in \eqref{error.b} leads to
\begin{equation}\label{def.falpha.basis}
f_\alpha=\sum_{i=1}^{N}d_i\eta_i.
\end{equation}
Since $\eta_i$ does not depend on $\alpha$ or $m$, this shows that any source reconstruction obtained by solving \eqref{weak.inverse} lies in the same finite dimensional space spanned by $\eta_i$.

\medskip
%
%The definition of $d$, \eqref{measurement}, \eqref{def.falpha.basis}, \eqref{aux}, \eqref{weak.forwardoperator} and \eqref{basis_eta} imply
%\begin{align*}
%\a d_i=m_i-(\Phi Tf_\a)_i&=m_i-<\phi_i,Tf_\a> \\
%&=m_i-<\phi_i,T(\sum_{j=1}^{N}d_j\eta_j)> \\
%%&=m_i-\sum_{j=1}^{N}<\phi_i,d_jT\eta_j>\\
%&=m_i-\sum_{j=1}^{N}d_ja(\xi_i,T\eta_j) \\
%&=m_i-\sum_{j=1}^{N}d_j (\xi_i,\eta_j) =m_i-\sum_{j=1}^{N}d_j b(\eta_i,\eta_j).
%\end{align*}
%Therefore, the coefficient vector $d$ can be obtained as
%$$d=(L+\a I)^{-1}m $$
%where $L$ is a symmetric and positive definite matrix defined by $L=[b(\eta_i,\eta_j)]_{1\le i,j\le N}$.
%
%\medskip

{{Let $P_N: F\to F$  be the orthogonal projection onto the span of $\{\eta_1,\cdots,\eta_N\}$ w.r.t. the inner product induced by the bilinear form $b(\cdot,\cdot)$. Thus
		\begin{equation}\label{def.projection}
		b(P_Nf_{\t},\eta_i)=b(f_{\t},\eta_i) \quad i=1,\cdots,N.
		\end{equation}
This orthogonal projection $P_N f_{\t}$ is the best approximation that can be obtained for $f_{\t}$ in the space which is the span of $\{\eta_1,\cdots,\eta_N\}$ where $f_\a$ belongs to.

\smallskip

To state the error estimate, let $L$ be the symmetric and positive definite matrix defined by $L=[b(\eta_j,\eta_i)]_{1\le i,j\le N}$
and $\max { \sigma{(L)}}:=\lambda_N$, where $\sigma(L)$ denotes the set of all eigenvalues of $L$. For obtaining an estimate for $\|f_{\t}-P_Nf_{\t}\|_k$, additional assumptions has to be imposed on $f_{\t}$.
		\begin{theorem}\label{thm.err}
			Let $f_{\t} \in F$ and $m=\Phi Tf_{\t}$. Let $f_{\a}$ be the reconstructed source field that solves \eqref{weak.inverse}. Then, %it holds that%$P_Nf_{\t}$ be a projection of $f_{\t}$ given by \eqref{def.projection} and 
			$$\|  f_{\t}-f_\a \|_k \le \| f_{\t}-P_Nf_{\t}\|_k +\frac{\alpha}{\lambda_N}\|f_{\t}\|_k.$$
		\end{theorem}
		\begin{proof}
	By triangle inequality, we have
			\begin{equation}\label{error.triangle}
			\|  f_{\t}-f_\a \|_k \le \| f_{\t}-P_Nf_{\t}\|_k +\|P_Nf_{\t} - f_\a \|_k.
			\end{equation}
	To estimate the last term in the right hand side of \eqref{error.triangle}, consider \eqref{def.projection}. From \eqref{basis_eta}, \eqref{weak.forwardoperator}, \eqref{aux} and \eqref{measurement}, we obtain
			\begin{equation}\label{bofft}
		b(f_{\t},\eta_i)=(\xi_i,f_{\t})=a(\xi_i, Tf_{\t})=(\phi_i,Tf_{\t})=(\Phi Tf_{\t})_i% \quad i=1,\cdots,N.
			\end{equation}
		for $i=1,\cdots,N.$ Since $	b(P_Nf_{\t},\eta_i)=b(f_{\t},\eta_i)$ for all $i=1,\cdots,N$, it follows from the above that $(\phi_i,Tf_{\t})=(\phi_i,TP_Nf_{\t})$ and hence
		%	Combine the above two estimates to write %the coefficient vector $c$ as
			\begin{equation}\label{PNf.f}%\label{def.a}
	\Phi TP_Nf_{\t}=m=\Phi Tf_{\t}.%	Lc=m.%a=L^{-1}\Phi Tf_{\t}=L^{-1}m.
			\end{equation}
		As a consequence, \eqref{weak.inverse} becomes
		\begin{align*}
		(\Phi Tf_\a)^*(\Phi Tg)+\alpha b(f_\a,g)&=(\Phi TP_Nf_{\t})^*(\Phi Tg) \quad \forall\, g \in F.
		\end{align*}
		Hence,
		\begin{equation}\label{eq.PNftfa}
		(\Phi T(P_Nf_{\t}-f_\a))^*(\Phi Tg)+\alpha b(P_Nf_{\t}-f_\a,g)=\alpha b(P_Nf_{\t},g).
		\end{equation}
Consider the right hand side of \eqref{eq.PNftfa}. Since the range of $P_N$ is the span of $\{\eta_1,\cdots,\eta_N\}$, writing $P_N g:=\sum_{i=1}^{N}c_i\eta_i$, by \eqref{def.projection},\eqref{bofft} and \eqref{PNf.f} , we have
\begin{align*}
b(P_Nf_{\t},g)&=b(P_Nf_{\t},P_N g)=\sum_{i=1}^{N}c_ib(P_Nf_{\t},\eta_i)\\
&=\sum_{i=1}^{N} c_i(\Phi TP_Nf_{\t})_i=c^*\Phi TP_Nf_{\t}=c^*\Phi Tf_{\t}.
\end{align*}
Therefore, \eqref{eq.PNftfa} can be expressed as
\begin{equation}\label{eq.PNftruefalpha}
	(\Phi T(P_Nf_{\t}-f_\a))^*(\Phi Tg)+\alpha b(P_Nf_{\t}-f_\a,g)=\alpha c^*\Phi Tf_{\t}.
\end{equation}
The coefficient vector $c$ of $P_N g$ is estimated now. Analogous arguments as in \eqref{bofft} leads to $ b(g,\eta_i)=(\Phi Tg)_i$ for $i=1,\cdots,N.$ Also, $b(P_Ng,\eta_i)=\sum_{j=1}^{N} c_jb(\eta_j,\eta_i)$. The definition of $P_N$ given by \eqref{def.projection} then implies
$$c=L^{-1}\Phi Tg.$$
%where $L$ is the symmetric and positive definite matrix defined by $L=[b(\eta_j,\eta_i)]_{1\le i,j\le N}$.
%\smallskip
%
This with \eqref{eq.PNftruefalpha} result in
\begin{align*}
	(\Phi T(P_Nf_{\t}-f_\a))^*(\Phi Tg)+\alpha b(P_Nf_{\t}-f_\a,g)&=\alpha(\Phi Tg)^*L^{-1}\Phi Tf_{\t}.
	%\\
%	&=\alpha (\Phi Tf_{\t},\Phi Tg).
\end{align*}	
In operator approach, the above equation reduces to
$$(\cA^*\cA+\alpha \cB)(P_Nf_{\t}-f_\a)=\alpha \cA^*L^{-1}\cA f_{\t}.$$
Since $\cA^*\cA+\alpha \cB$ is invertible (\eqref{operator}-\eqref{bound.falpha}), we get
$$\|P_Nf_{\t}-f_\a \|_k=\frac{\alpha}{ \lambda_N}\|(\cA^*\cA+\alpha \cB)^{-1}\cA^*\cA f_{\t} \|_k\le \frac{\alpha}{\lambda_N}\|f_{\t}\|_k. $$
The combination of this and \eqref{error.triangle} concludes the proof.
		\end{proof}
		Suppose the measurement $m$ is noisy. That is, for $\delta>0$, we may have $m^\delta$ in place of $m$ such that $\|m-m^\delta\| \le \delta$.%
		
		\begin{theorem}\label{thm.f-falphadelta}Let $f_{\t} \in F$ and $m=\Phi Tf_{\t}$. Let $P_Nf_{\t}$ be as in \eqref{def.projection} and $f_{\a}^\delta$ solves \eqref{weak.inverse} with noisy measurement $m^\delta$. Then, 
		\begin{equation}\label{error.conts}
	\|  f_{\t}-f_\a^\delta \|_k \le \| f_{\t}-P_Nf_{\t}\|_k +\frac{\alpha}{\lambda_N}\|f_{\t}\|_k+\frac{\delta}{2\sqrt{\alpha}}.
		\end{equation}
%	In particular, choosing $\sqrt{\alpha}=\argmin_{\alpha}\{\frac{\alpha}{\lambda_N}\|f_{\t}\|_k+\frac{\delta}{2\sqrt{\alpha}}\}$, we have
%$$\|  f_{\t}-f_\a^\delta \|_k \le \| f_{\t}-P_Nf_{\t}\|_k +\sqrt{\alpha}.$$			
		\end{theorem}
		\begin{proof}
			The triangle inequality leads to
			\begin{equation}\label{err.tri1}
			\|  f_{\t}-f_\a^\delta \|_k \le \|  f_{\t}-f_{\a}\|_k+\|f_{\a} - f_\a^\delta  \|_k.
			\end{equation}
			Given the true source field $f_{\t}$, the noisy measurement $m^\delta$ with additive noise $n$ is obtained as
			$$ m^\delta=m+n$$
			where $m=\Phi Tf_{\t}$ and $\|n\| \le \delta$. The definition of \eqref{operator} and spectral theory \cite[Corollary 4.4]{MTN_book_LOE} show that
			\begin{align*}
		\|f_{\a} - f_\a^\delta \|_k &=\|(\mathcal{A^*A}+\alpha \mathcal{B})^{-1}\mathcal{A^*}m-(\mathcal{A^*A}+\alpha \mathcal{B})^{-1}\mathcal{A^*}m^\delta\|_k \le \frac{\delta}{2\sqrt{\alpha}}.
			\end{align*}
%			
%			From \eqref{def.falpha.basis}, $f_\alpha=\sum_{i=1}^{N}d_i\eta_i$ where $d=(L+\alpha I)^{-1}m$. In a similar way, we can write $f_\alpha^\delta=\sum_{i=1}^{N}d_i^\delta\eta_i$ where $d^\delta=(L+\alpha I)^{-1}m^\delta$. 
%			
%			Use the coercivity of the bilinear form $b(\cdot,\cdot)$ to obtain
%			\begin{align*}
%			C\| f_\a-f_\a^\delta \|^2 &\le b(f_\a-f_\a^\delta,f_\a-f_\a^\delta)\\
%			&=(d-d^\delta)^*L(d-d^\delta)\\
%			&=n^*((L+\alpha I)^{-1})^* L ((L+\alpha I)^{-1})n\\
%			&\le \max_{\lambda \in \sigma{(L)}}\frac{\lambda}{(\alpha+\lambda)^2}\|n\|^2\le\frac{1}{2\alpha}\|n\|^2.
%			\end{align*}
%			Therefore,
%			$$\| f_\a-f_\a^\delta \| \le \frac{C\delta}{\sqrt{2\alpha}},$$
%			where $C$ depends on the coercivity of $b(\cdot,\cdot)$. 
The combination of Theorem \ref{thm.err} and the above inequality leads to the desired estimate in \eqref{error.conts}. %The particular case is obvious from \eqref{error.conts}.
		\end{proof}}}	
\begin{remark}\label{rem.regularization}
By choosing the regularization parameter $\alpha$ appropriately, depending on the data error $\delta$, we obtain an estimate for $\|f_{\rm true} - f_\alpha^\delta\|_k$ of the order of  $\|f_{\rm true} - P_Nf_{\rm true}\|_k$. From the expression on right hand side of \eqref{error.conts}, it is clear that an optimal choice of $\alpha$ depending on $\delta$ for a fixed $N$ is $ \alpha \sim  \delta^{2/3}$. For this choice, the estimate in \eqref{error.conts} leads to 
$$\|f_{\rm true} - f_\alpha^\delta\|_k 
\leq  
\|f_{\rm true} - P_Nf_{\rm true}\|_k + C \delta^{2/3}$$
so that for small enough data error $\delta$, we obtain, 
$$\|f_{\rm true} - f_\alpha^\delta\|_k \leq  C \|f_{\rm true} - P_Nf_{\rm true}\|_k.$$
\end{remark} 
 We are interested in the numerical approximation of the reconstructed source field $f_\a$ and the error estimate which are discussed below.
%		 It is interesting to note the following estimate for $\|f_{\t}-P_Nf_{\t}\|_k$ even though additional assumptions on $f_{\t}$ is required to ensure the convergence and this is not the main focus of this article. The coercivity and continuity properties of $b(\cdot,\cdot)$, $b(f_{\t},P_Nf_{\t})=b(P_Nf_{\t},P_Nf_{\t})$ from \eqref{def.projection} (since the range of $P_N$ is the span of $\{\eta_1,\cdots,\eta_N\}$) and \eqref{weak.inverse} imply
%\begin{align*}
%C_2\|f_{\t}-&P_Nf_{\t}\|_k^2\le b(f_{\t}-P_Nf_{\t},f_{\t}-P_Nf_{\t})\\
%&=b(f_{\t}-P_Nf_{\t},f_{\t}-f_\a)+b(f_{\t}-P_Nf_{\t},f_\a)\\
%&\le C_1\|f_{\t}-P_Nf_{\t} \|_k\|f_{\t}-f_\a \|_k+\frac{1}{\alpha}\left((\Phi T(f_{\t}-f_{\a}))^*\Phi T ((f_{\t}-P_Nf_{\t})\right)\\
%&= C_1\|f_{\t}-f_{\a} \|_k \|f_{\t}-P_Nf_{\t} \|_k%\le(1+\frac{C(N)}{\alpha})\|f_{\t}-f_{\a} \|_k \|f_{\t}-P_Nf_{\t} \|_k.
%\end{align*}
%with \eqref{PNf.f} in the end. Consequently, $\|f_{\t}-P_Nf_{\t}\|_k\le C\|f_{\t}-f_{\a} \|_k$.}}
%\smallskip
%
% We are interested in the numerical approximation of the reconstructed source field $f_\a$ and the error estimate which are discussed below.
 \subsection{Discretisation of the inverse problem}
This section deals with the discretisation of the inverse problem \eqref{weak.inverse} which solves $f_\a$. The function $f_\a$ is in the infinite-dimensional space $F$, which is
typically impossible to completely handle in a computer based environment. We
need to come down to a finite-dimensional space by replacing $F$ with a finite dimensional space, which is not necessarily a subspace of $F$.

\medskip

 Let $\cT_\tau$ be a regular triangulation of $\overline{\O}$ with mesh size $\tau$ (see Section \ref{sec:fem} for details). We discretise $F$ with a piecewise polynomial finite element space $F^\tau$. For approximation using conforming FEMs,
 $$F^\tau:=F^\C=\{f \in F\,| \,f_{|K} \in \cP_q(K) \quad \forall\, K \in \cT_\tau\}$$ 
 and for nonconforming FEMs,
  \[
 F^\tau:=F^\NC=\left \{ \begin{aligned}&f \in L^2(\O)\,| \,f_{|K} \in \cP_q(K) \quad \forall\, K \in \cT_\tau,
 \text{degrees of freedom consist of} \\
 & \text{point values of the functions and its derivatives upto
 	certain order}
 \end{aligned}
 \right\}
 \]
 where $\cP_q(K)$ is the space of polynomials of degree $\le q$ on element $K$. The following are a few special cases:
 \begin{itemize}
 	\item For $k=0$, $F^\tau \subset L^2(\O)=F$.
 	\item For $k=1$ with $F=H^1_0(\O)$, one example of $F^\C$ (resp. $F^\NC)$ is given by the conforming (resp. nonconforming) $\cP_1$ finite element space \cite{ciarlet1978finite,CCDGMS}.
\item For $k=2$ with $F=V$, we can take $F^\C=V_\C$, the Argyris or the \BFS finite element and $F^\NC=V_\M$, the Morley finite element space.
 \end{itemize}
 Define $\|\cdot\|_{k,\tau}^2:=\|D_{\rm NC}^k\cdot\|^2$, where $D^k_\NC$ denotes the piecewise $k^{th}$ derivatives, and for conforming FEMs, $D^k_\NC=D^k$.
 \begin{assumption}[Interpolation of the reconstruction space]\label{assumption.inverseinterpolation} There exists an interpolation operator $I^\tau:F \to F^\tau$ such that
 	$$\| g-I^\tau g \|_{k,\tau} \le C\tau^{\min\{q+1,t\}-k}\|g\|_t \quad \forall\, g \in F \cap H^t(\O),$$
 	where $C>0$ is independent of $\tau$ and $g$.
 \end{assumption}
The well-known property of the interpolation operator $I^\tau$ in the above assumption can be found in \cite{ciarlet1978finite}. The particular case with $F=V$ and $F^\tau=V_\C$ (resp. $V_\M$) is stated in Lemma \ref{interpolant.cfem} (resp. Lemma \ref{Morley_Interpolation}).

\smallskip

The discrete problem corresponding to the weak formulation \eqref{weak.inverse} seeks 
$f_{\alpha,h}^\tau\in F^\tau$ such that
\begin{equation}\label{weak.inverse.modified}
B^\tau(f_{\alpha,h}^\tau,g^\tau)=l^\tau(g^\tau) \quad \forall\, g^\tau \in F^\tau,
\end{equation}
where for all $f^\tau, g^\tau \in F^\tau$,
\begin{equation}\label{weak.inverse.modified.Bandl}
B^\tau(f^\tau,g^\tau)=(\Phi T_h f^\tau)^*(\Phi T_h g^\tau)+\alpha b^\tau(f^\tau,g^\tau) \mbox{ and } l^\tau(g^\tau)=m^*(\Phi T_h g^\tau).
\end{equation}
The only difference of \eqref{weak.inverse.modified} with \eqref{weak.inverse} is that the operator $T$ (resp. the continuous bilinear form $b(\cdot,\cdot)$) is replaced by $T_h$ (resp. the discrete bilinear form $b^\tau(\cdot,\cdot)$) and hence computable.

\begin{assumption}\label{assumption.companion}
	We assume that the discrete bilinear form $b^\tau(\cdot,\cdot)$ is continuous in $F+F^\tau$ and coercive in $F^\tau$ with respect to the norm $\|\cdot\|_{k,\tau}$. Also, there exists an operator $E^\tau:F^\tau \rightarrow F$ such that, for $f^\tau,g^\tau \in F^\tau$,
\begin{itemize}
	\item[$(a)$] 	$\sup_{\| f^\tau\|_{k,\tau}=1}\| f^\tau-E^\tau f^\tau\| = \delta_6$
	\item[$(b)$] $\sup_{\| f^\tau\|_{k,\tau}=1}	\| f^\tau-E^\tau f^\tau\|_{k,\tau}= \delta_7$
\item[$(c)$]	 $\sup_{\| f^\tau\|_{k,\tau}=1}\sup_{\| g^\tau\|_{k,\tau}=1}b^\tau(f^\tau-E^\tau f^\tau, g^\tau) =\delta_8$,
\end{itemize}
where $\delta_6,\delta_7,\delta_8$ are non-negative parameters.	
	%$$|b(I^\tau f_\alpha,I^\tau f_\alpha-f_{\alpha,{\rm M}}^\tau)-b_{\rm NC}(I^\tau f_\alpha,I^\tau f_\alpha-f_{\alpha,{\rm M}}^\tau)| \le Ch^{t}.$$
\end{assumption}
For conforming FEMs, $F^\tau \subset F$ and $b^\tau(\cdot,\cdot)=b(\cdot,\cdot)$. In this case, $E^\tau=\rm{Id}$  and hence the above assumption is trivially satisfied. Assumption \ref{assumption.companion} is discussed in details for conforming and nonconforming FEMs in Section \ref{sec.feminverse}.

\smallskip

The proof of the next lemma follows similar to Lemma \ref{lem.bilinearB} and hence it is skipped.
\begin{lemma}\label{lem.bilinearBh}
	The bilinear form $B^\tau(\cdot,\cdot)$ is continuous and coercive in $F^\tau$.
\end{lemma}
%\begin{proof}
By Lax-Milgram Lemma, there exists a unique solution $f_{\alpha,h}^\tau$ to \eqref{weak.inverse.modified} and the coercivity result of $B^\tau(\cdot,\cdot)$ shows $\|f_{\alpha,h}^\tau\|_{k,\tau} \lesssim \|m\|$.

\subsection{Error estimate for the reconstructed regularised solution}
This section deals with the error estimate between the reconstructed regularised approximation of source field $f_\alpha \in F$ and its numerical approximation $f_{\alpha,h}^\tau \in F^\tau$ for different regularisation schemes, i.e. for $k \in \N_0$.

\smallskip

The solution $f_\a$ to \eqref{weak.inverse} will often have more regularity than $H^k(\O)$, which is stated in the following assumption, but the actual regularity is problem specific.

\smallskip

\begin{assumption}[Regularity of the source field reconstruction]\label{assumption.f_reg}
	There is an $\ell\ge k$ such that the solution $f_\alpha$ to \eqref{weak.inverse} is in $H^{\ell}(\O)$ and it holds that
	$$\|f_\alpha\|_{\ell} \le C\|m \|,$$
	where $C>0$ is independent of $m$.	
\end{assumption}
From \eqref{def.falpha.basis}, it is observed that the regularity of $f_\a$ is determined by the regularity of the reconstruction basis functions $\eta_i$ which is the solution of \eqref{basis_eta} ($b(\eta_i,g)=(\xi_i,g) \quad \forall\, g \in F.$). Consequently, the regularity of $f_\a$ depends on the regularity of $\xi_i$ and the bilinear form $b(\cdot,\cdot)$. For instance, if $b(f,g)=(f,g)$ with $k=0$, then $\eta_i=\xi_i$ and a combination of \eqref{def.falpha.basis} and Assumption \ref{assumption} yields $f_\a \in  H^r(\O)$.

To simplify the notation, define $$\delta_9:=\tau^{\min\{q+1,\ell\}-k}. $$
Recall the notation $\widetilde{\beta}(r,s)$ defined in \eqref{def.widetilde}. %{\color{blue}{Notice that the parameter $\alpha$ is fixed for a particular problem.}}
\begin{theorem}[Total error]\label{thm.totalerror}
	Let $f_\a$ be the solution to \eqref{weak.inverse} and $f_{\a,h}^\tau$ be the solution to \eqref{weak.inverse.modified}. Suppose Assumption \ref{assumption.regularity.forward} holds with $s$, Assumption \ref{assumption.forwardoperators}, Assumption \ref{assumption.f_reg} (resp. \ref{assumption}) holds with $\ell$ (resp. $r$), Assumptions \ref{assumption.inverseinterpolation} and \ref{assumption.companion} hold. Then there exists a constant $C>0$ independent of $m, h$ and $\tau$ such that
	$$\|f_\alpha-f_{\alpha,h}^\tau \|_{k,\tau}\le C(\alpha,N) \big(\delta_6+\delta_8(1+\delta_9)+(1+\delta_7)(\delta_9+\widetilde{\beta}(r,s) \big)\|m\|.$$	
\end{theorem}
\begin{proof}
The triangle inequality with $I^\tau f_{\alpha}$ and Assumption \ref{assumption.inverseinterpolation} lead to
\begin{align}
\|f_\alpha-f_{\alpha,h}^\tau \|_{k,\tau} &\le \|f_\alpha-I^\tau f_{\alpha} \|_{k,\tau} +\|I^\tau f_{\alpha}-f_{\alpha,h}^\tau \|_{k,\tau}\nonumber\\
&\lesssim \delta_9\|f_\alpha\|_\ell+\|I^\tau f_{\alpha}-f_{\alpha,h}^\tau \|_{k,\tau}.\label{trig}
\end{align}
Consider $\|f^\tau\|_{k,\tau}:= \|I^\tau f_\alpha-f_{\alpha,h}^\tau \|_{k,\tau}$. Notice that $f^\tau \in F^\tau$. The coercivity and continuity of $B^\tau(\cdot,\cdot)$ in Lemma \ref{lem.bilinearBh}, Assumption \ref{assumption.inverseinterpolation} and simple manipulation show
\begin{align}
\alpha C_2^\tau\|f^\tau \|_{k,\tau}^2\le B^\tau(f^\tau, f^\tau )&=B^\tau(I^\tau f_{\alpha}-f_{\alpha} ,f^\tau )+B^\tau(f_{\alpha}-f_{\alpha,h}^\tau ,f^\tau )\nonumber \\
&\le (C(N)+\alpha C_1^\tau) \| I^\tau f_{\alpha}-f_{\alpha}\|_{k,\tau} \|f^\tau\|_{k,\tau}\nonumber\\
&\qquad +B^\tau(f_{\alpha}-f_{\alpha,h}^\tau ,f^\tau )\nonumber\\
&\le C(C(N)+\alpha C_1^\tau) \delta_9\|f_\alpha\|_\ell \| f^\tau\|_{k,\tau}\nonumber\\
&\qquad +B^\tau(f_{\alpha}-f_{\alpha,h}^\tau ,f^\tau ),\label{trig1}
\end{align}
where $C_1^\tau$ (resp. $C_2^\tau$) denotes the continuity (resp. coercivity) constant of $b^\tau(\cdot,\cdot)$. Since $f^\tau \in F^\tau$, the last term in the right hand side of \eqref{trig1} can be rewritten as
\begin{align}\label{t1t2t3}
B^\tau(f_{\alpha}-f_{\alpha,h}^\tau ,f^\tau )&=B^\tau(f_{\alpha} ,f^\tau - E^\tau f^\tau )+B^\tau(f_{\alpha} , E^\tau f^\tau )-B(f_{\alpha} , E^\tau f^\tau )\nonumber\\
&\quad +B(f_{\alpha} , E^\tau f^\tau )-B^\tau(f_{\alpha,h}^\tau ,f^\tau )\nonumber\\
&=B^\tau(f_{\alpha} ,f^\tau - E^\tau f^\tau )+(B^\tau(f_{\alpha} , E^\tau f^\tau )-B(f_{\alpha} , E^\tau f^\tau ))\nonumber\\
&\quad +(l(E^\tau f^\tau )-l^\tau(f^\tau ))=:T_1+T_2+T_3
%=B^\tau(f_{\alpha} ,f^\tau -E^\tau f^\tau)+B^\tau(f_{\alpha} ,E^\tau f^\tau)-B^\tau(f_{\alpha,h}^\tau ,f^\tau )\nonumber\\
%&=B^\tau(f_{\alpha} ,f^\tau -E^\tau f^\tau)-B(f_{\alpha} ,f^\tau -E^\tau f^\tau)+B^\tau(f_{\alpha} ,E^\tau f^\tau)\nonumber\\
%&\qquad -B(f_\alpha, E^\tau f^\tau)+B(f_\alpha,f^\tau)-B^\tau(f_{\alpha,h}^\tau ,f^\tau )\nonumber\\
%&=\big(B^\tau(f_{\alpha} ,f^\tau -E^\tau f^\tau)-B(f_{\alpha} ,f^\tau -E^\tau f^\tau)\big)\nonumber\\
%&\qquad +\big(B^\tau(f_{\alpha} ,E^\tau f^\tau)-B(f_\alpha, E^\tau f^\tau)\big)+ \big(l(f^\tau)-l_\NC(f^\tau)\big)
\end{align}
with \eqref{weak.inverse} and \eqref{weak.inverse.modified} in the last step. From Assumptions \ref{assumption.companion} $(c)$ and \ref{assumption.inverseinterpolation}, $ b^\tau(I^\tau f_\alpha, f^\tau - E^\tau f^\tau)\le C\delta_8(\delta_9+1)\|f_\alpha\|_\ell\|f^\tau\|_{k,\tau}.$ This, the definition of $B^\tau(\cdot,\cdot)$, a Cauchy-Schwarz inequality, \eqref{stability.measurement}, Assumptions \ref{assumption.inverseinterpolation},  \ref{assumption.companion} $(a).(b)$ and the continuity of $b^\tau(\cdot,\cdot)$ show that
\begin{align}
T_1&\le (\Phi T_h f_\alpha)^*\Phi T_h(f^\tau - E^\tau f^\tau)+\alpha b^\tau(f_\alpha-I^\tau f_\alpha, f^\tau - E^\tau f^\tau)\nonumber\\
&\qquad + C\delta_8(\delta_9+1)\|f_\alpha\|_\ell\|f^\tau\|_k\nonumber\\
%&=(\Phi T_h f_\alpha)^*\Phi T_h(f^\tau - E^\tau f^\tau)+\alpha b^\tau(f_\alpha-I^\tau f_\alpha, f^\tau - E^\tau f^\tau)+ \delta_8\|f_\alpha\|_k\|f^\tau\|_k \nonumber\\
&\le C\|f_\alpha\|\|f^\tau - E^\tau f^\tau\|+\alpha C^\tau_1 \| f_\alpha-I^\tau f_\alpha\|_{k,\tau} \|  f^\tau - E^\tau f^\tau\|_{k,\tau}\nonumber\\
&\qquad+ C\delta_8(\delta_9+1)\|f_\alpha\|_\ell\|f^\tau\|_{k,\tau} \nonumber\\
%&\le C \big(\delta_6+ \alpha\delta_9\delta_7+\delta_8(\delta_9+1)\big)\|f_\alpha\|_\ell \|f^\tau\|_{k,\tau}\nonumber\\
&\le C(\alpha,N) \big(\delta_6+ \delta_9\delta_7+\delta_8(\delta_9+1)\big)\|f_\alpha\|_\ell\|f^\tau\|_{k,\tau}.\nonumber%\label{t1}
\end{align}
Since $f_\alpha, E^\tau f^\tau \in F$, $b^\tau(f_\alpha, E^\tau f^\tau)=b(f_\alpha, E^\tau f^\tau)$. The triangle inequality with $f^\tau$ and Assumption \ref{assumption.companion} $(b)$ read $\|E^\tau f^\tau\|_k \le (\delta_7+1)\|f^\tau\|_{k,\tau}$. This, the definition of $B^\tau(\cdot,\cdot)$ and $B(\cdot,\cdot)$, Cauchy-Schwarz inequalities, Theorem \ref{thm.forwarderror} and \eqref{stability.measurement} prove
\begin{align}
T_2&=(\Phi T_h f_\alpha)^* \Phi T_h (E^\tau f^\tau)-(\Phi T f_\alpha)^* \Phi T (E^\tau f^\tau)\nonumber\\
&=(\Phi (T_h-T)f_\alpha)^* \Phi T_h (E^\tau f^\tau)+(\Phi T f_\alpha)^*\Phi (T_h-T) (E^\tau f^\tau)\nonumber\\
%&\quad +(\Phi (T_h-T)f_\alpha)^*\Phi (T_h-T) (E^\tau f^\tau)
&\le \|\Phi (T_h-T)f_\alpha\|\| \Phi T_h (E^\tau f^\tau)\|+\|\Phi T f_\alpha\|\|\Phi (T_h-T) (E^\tau f^\tau)\|\nonumber\\
&\le C(N) (\delta_7+1)\widetilde{\beta}(r,s)\|f_\alpha\|_k\|f^\tau\|_{k,\tau}.\nonumber
\end{align}
Theorem \ref{thm.forwarderror}, \eqref{stability.measurement} and Assumption \ref{assumption.companion} $(a). (b)$ result in
\begin{align}
T_3&=m^*(\Phi (T-T_h)(E^\tau f^\tau)-\Phi T_h(f^\tau-E^\tau f^\tau ))\nonumber\\
&\le \|m\|\big(\|\Phi (T-T_h)(E^\tau f^\tau) \|+\| \Phi T_h(f^\tau-E^\tau f^\tau )\|\big)\nonumber\\
&\le C(N) ((\delta_7+1)\widetilde{\beta}(r,s)+\delta_6)\|m\|\|f^\tau\|_{k,\tau}.\nonumber
\end{align}
The combination of $T_1-T_3$ in \eqref{t1t2t3} and then in \eqref{trig1} lead to
$$\|I^\tau f_{\alpha}-f_{\alpha,h}^\tau \|_k\le C(\alpha,N)\big( \delta_6+\delta_8(1+\delta_9)+(1+\delta_7)(\delta_9+\widetilde{\beta}(r,s))\big)\|m\|$$
with Assumption \ref{assumption.f_reg} in the end. This and \eqref{trig} conclude the proof.
\end{proof}
The analogous arguments as in Theorem \ref{thm.totalerror} lead to the following result.

\begin{corollary}\label{cor.thmnoisy}
	Suppose the measurement is noisy. Let $f_\a^\delta$ (resp. $f_{\a,h}^{\delta,\tau}$) be the solution to \eqref{weak.inverse} (resp. \eqref{weak.inverse.modified}) with $m$ replaced by $m^\delta$. Then under the Assumptions of Theorem \ref{thm.totalerror}, the following error estimate hold:
	$$\|f_\alpha^\delta-f_{\alpha,h}^{\delta,\tau} \|_k\le  C \big(\delta_6+\delta_8(1+\delta_9)+(1+\delta_7)(\delta_9+\widetilde{\beta}(r,s) \big)\|m^\delta\|,$$
	where $C>0$ is a constant independent of $m^\delta, h$ and $\tau$. 
\end{corollary}
The result of next theorem follows from the triangle inequality, Corollary \ref{cor.thmnoisy} and Theorem \ref{thm.f-falphadelta}.
\begin{theorem}\label{thm.finalerror}
	Let $f_{\t} \in F$ and $P_Nf_{\t}$ be a projection of $f_{\t}$ given by \eqref{def.projection}. Suppose the measurement is noisy. Let $f_{\a,h}^{\delta,\tau}$ be the solution to \eqref{weak.inverse.modified} with $m$ replaced by $m^\delta$. Then, under the assumptions of Theorem \ref{thm.totalerror}, there exists a constant $C>0$ independent of $h$, $\tau$and $m$ such that
	\begin{align*}
\|  f_{\t}-f_{\alpha,h}^{\delta,\tau} \|_k &\le \| f_{\t}-P_Nf_{\t}\|_k +\frac{\alpha}{\lambda_N}\|f_{\t}\|_k+\frac{\delta}{2\sqrt{\alpha}}\\
&\qquad +C\big(\delta_6+\delta_8(1+\delta_9)+(1+\delta_7)(\delta_9+\widetilde{\beta}(r,s) \big)\|m^\delta\|.
	\end{align*}
\end{theorem}
\begin{remark}
For sufficiently small $\delta_i,\, i=1,\cdots,9$, the combination of Remark \ref{rem.regularization} and Theorem \ref{thm.finalerror} imply
$$\|f_{\rm true} - f_\alpha^\delta\|_k \leq  C \|f_{\rm true} - P_Nf_{\rm true}\|_k.$$
\end{remark}
 
\subsection{Application to conforming and nonconforming FEMs}\label{sec.feminverse}
The application of the results stated in the previous section when the solution $f_\a $ is approximated by the conforming ($F^\tau=F^\C$) and nonconforming ($F^\tau=F^\NC$) FEMs are discussed in this section.
\subsubsection{Conforming FEMs}\label{sec.cfem.totalerror}
For conforming FEMs, we have 
$$F^\tau \subset F,\,b^\tau(\cdot,\cdot)=b(\cdot,\cdot)\mbox{ and }\|\cdot \|_{k,\tau}=\|\cdot\|_k.$$
 In this case, $E^\tau=\rm{Id}$  and hence Assumption \ref{assumption.companion} is trivially satisfied with $\delta_6=\delta_7=\delta_8=0$. Thus, Theorem \ref{thm.totalerror} result in
\begin{equation}\label{thm.totalerror.cfem}
\|f_\alpha-f_{\alpha,h}^\tau \|_{k}\le C(\alpha,N) \big( \tau^{\min\{q+1,\ell\}-k}+\widetilde{\beta}(r,s)\big)\|m\|.
\end{equation}	
Recall $\widetilde{\beta}(r,s)$ defined in \eqref{def.widetilde} is the parameter appearing in the approximation error in Theorem \ref{thm.forwarderror} of the forward problem that employs $V_h$ to discretise the displacement $u$. 
\begin{itemize}
	\item If $V_h=V_\C$, then \eqref{thm.totalerror.cfem},\eqref{def.widetilde} and \eqref{thm.forwarderror.cfem} yield 
	\begin{equation}\label{thm.totalerror.cfemvc}
	\|f_\alpha-f_{\alpha,h}^\tau \|_{k}\le C(\alpha,N) \big( \tau^{\min\{q+1,\ell\}-k}+h^{\min\{p+1,\;s\}+\min\{p+1,\;r\}-4}\big)\|m\|.
	\end{equation}		
\item If $V_h=V_\M$, then a combination of \eqref{thm.totalerror.cfem}, \eqref{def.widetilde} and \eqref{thm.forwarderror.ncfem}  leads to
	\begin{equation*}\label{thm.totalerror.cfemvm}
\|f_\alpha-f_{\alpha,h}^\tau \|_{k}\le C(\alpha,N) \big( \tau^{\min\{q+1,\ell\}-k}+h^{\min\{1,s-2\}+\min\{1,r-2\}}\big)\|m\|.
	\end{equation*}	
\end{itemize}
In particular, if $f \in L^2(\O)$ with $u,\xi \in H^{2+\gamma}(\O)$ and $q+1\ge \ell$, then the above result reduces to
$$	\|f_\alpha-f_{\alpha,h}^\tau \|\lesssim \big(\tau^{\ell-k}+h^{2\gamma}\big)\|m\|.$$

%If $V_h=V_\C$ (resp. $V_\M$), the above inequality, \eqref{def.widetilde} and the results in Section \ref{sec.cfem.forward} (resp. Section \ref{sec.ncfem.forward}) show that
%$$\|f_\alpha-f_{\alpha,h}^\tau \|_{k}\le C(\alpha,N) \big( \tau^{\min\{q+1,\ell\}-k}+h^{\min\{p+1,\;s\}+\min\{p+1,\;r\}-4}\big)\|m\|.$$
%$$\mbox{(resp. }\|f_\alpha-f_{\alpha,h}^\tau \|_{k}\le C(\alpha,N) \big( \tau^{\min\{q+1,\ell\}-k}+h^{\min\{1,s-2\}+\min\{1,r-2\}}\big)\|m\|.)$$
\subsubsection{Nonconforming FEMs}\label{sec.ncfem.totalerror}
Some examples of nonconforming FEMs for the cases $k=1$ and $k=2$ are provided in this section. Compared with conforming finite elements, the nonconforming finite elements employ fewer degrees of freedom and hence are attractive. 

\smallskip

$\bullet$ {\bf{$k=1:$}} For $F=H^1_0(\O)$ and $b(f,g)=\int_\O \nabla f\cdot\nabla g \dx$ for all $f, g \in F$, choose $F^\tau$ as the nonconforming $P_1$ finite element space ($q=1$). The discrete bilinear form can be defined as the sum of the elementwise integral of the gradient of the functions, that is, $$b^\tau(f^\tau,g^\tau)=\sit \nabla f^\tau\cdot \nabla g^\tau \dx \quad\forall\, f^\tau, g^\tau \in F^\tau.$$ 
It can be checked that $b^\tau(\cdot,\cdot)$ is continuous in $F+F^\tau$ and coercive in $F^\tau$ with respect to the norm $\|\cdot\|_{k,\tau}:=\|\cdot\|_{1,\tau}$.
There exists an operator $E^\tau$ with the required property given by $\delta_6\lesssim \tau$, $\delta_7\lesssim 1$ and $\delta_8=0$ \cite[Proposition 2.3]{CCDGMS}. Thus, Assumption \ref{assumption.companion} holds true. Also, Assumption \ref{assumption.inverseinterpolation} is satisfied \cite{ciarlet1978finite}. Hence Theorem \ref{thm.totalerror} implies
\begin{equation*}\label{thm.totalerror.ncfemk1}
	\|f_\alpha-f_{\alpha,h}^\tau \|_{1,\tau}\le C(\alpha,N) \big( \tau^{\min\{1,\ell-1\}}+\widetilde{\beta}(r,s)\big)\|m\|.
	\end{equation*}
	As in Section \ref{sec.cfem.totalerror}, for $V_h=V_\C$ (resp. $V_h=V_\M$), the above estimate can be rewritten as
	$$\|f_\alpha-f_{\alpha,h}^\tau \|_{1,\tau}\le C(\alpha,N) \big( \tau^{\min\{1,\ell-1\}}+h^{\min\{p+1,\;s\}+\min\{p+1,\;r\}-4}\big)\|m\|.$$
	$$ (\mbox{ resp. } \|f_\alpha-f_{\alpha,h}^\tau \|_{1,\tau}\le C(\alpha,N) \big( \tau^{\min\{1,\ell-1\}}+h^{\min\{1,s-2\}+\min\{1,r-2\}}\big)\|m\|).$$
	
 \smallskip

 $\bullet$ {\bf{$k=2:$}} For $F=V$ and $b(\cdot,\cdot)=a(\cdot,\cdot)$, choose $F^\tau=V_\M$ ($q=2$). Then
 $$b^\tau(\cdot,\cdot)=a_h(\cdot,\cdot), I^\tau=I_\M, E^\tau=E_\M, \delta_6\lesssim \tau^2, \delta_7\lesssim1 \mbox{ and } \delta_8=0,$$ 
we refer Section \ref{sec.ncfem.forward} for more details. As a result, Assumptions \ref{assumption.inverseinterpolation} and \ref{assumption.companion} hold and Theorem \ref{thm.totalerror} shows
\begin{equation}\label{thm.totalerror.ncfemk2}
\|f_\alpha-f_{\alpha,h}^\tau \|_{2,\tau}\le  C(\alpha,N) \big( \tau^{\min\{1,\ell-2\}}+\widetilde{\beta}(r,s)\big)\|m\|.
\end{equation}
Consequently, 
for $V_h=V_\C$ (resp. $V_h=V_\M$), the above result simplifies to
$$\|f_\alpha-f_{\alpha,h}^\tau \|_{2,\tau}\le C(\alpha,N) \big( \tau^{\min\{1,\ell-2\}}+h^{\min\{p+1,\;s\}+\min\{p+1,\;r\}-4}\big)\|m\|.$$
	\begin{equation}\label{thm.totalerror.ncfemvm}(\mbox{ resp. } \|f_\alpha-f_{\alpha,h}^\tau \|_{2,\tau}\le C(\alpha,N) \big( \tau^{\min\{1,\ell-2\}}+h^{\min\{1,s-2\}+\min\{1,r-2\}}\big)\|m\|).
	\end{equation}
%	\begin{remark}
%	The Adini element is a well-known nonconforming finite element for biharmonic equations on rectangular meshes with 12 degrees of freedom in a rectangle. There exists an enrichment operator $E^\tau$ which satisfies the properties $(a).(b)$ in Assumption \ref{assumption.companion} for $k=2$, we refer to\cite{Brenner_ncfem} for more details. Hence the result in Theorem \ref{thm.totalerror} is applicable for Adini nonconforming FEM.
%	\end{remark}
%If $V_h=V_\C$ (resp. $V_\M$), then \eqref{def.widetilde} and Section \ref{sec.cfem.forward} (resp. Section \ref{sec.ncfem.forward}) yield $\widetilde{\beta}(r,s)=h^{\min\{p+1,\;s\}+\min\{p+1,\;r\}-4}$ (resp. $\widetilde{\beta}(r,s)=h^{\min\{1,s-2\}+\min\{1,r-2\}}).$
\section{Numerical Section}\label{sec.numericalsection}
This section is devoted to the implementation procedure to solve the discrete inverse problem and is followed by numerical results for the BFS conforming FEM and Morley nonconforming FEM. 

\subsection{Implementation procedure}\label{sec.implementation}
The implementation procedure to solve the discrete inverse problem \eqref{weak.inverse.modified} is described in this section.

\smallskip
Let $u_h \in V_h$ solves \eqref{weak.forward.fem} and $f_{\alpha,h}^\tau \in F^\tau$ solves \eqref{weak.inverse.modified}. Let $\{\phi_1,\cdots,\phi_{{m}_{1}}\}$ be the global basis functions of $V_h$ and  $\{\psi_1,\cdots,\psi_{{m}_{2}}\}$ be the global basis functions of $F^\tau$. Then 
$T_hf=u_h:=\sum_{i=1}^{{m}_{1}}u_i\phi_i$ and $f_{\alpha,h}^\tau:=\sum_{k=1}^{{m}_{2}}f_k\psi_k$, where $u_i$ and $f_k$ for $i=1,\cdots,{m}_{1}$ and $k=1,\cdots,{m}_{2}$ are the unknowns. Substituting this in \eqref{weak.forward.fem} and choosing $v_h=\phi_j$, we obtain the matrix equation of the discrete forward problem as
$$\bA\bu=\bF$$ where 
	$$\bA:=\left[\int_{\Omega}\Delta \phi_i\Delta \phi_j \dx\right]_{1\le i,j\le m_1},\,\bF=\left[\int_{\Omega}f\phi_j \dx\right]_{1\le j\le m_1} \mbox{ and } \bu=[u_i]_{1\le i\le m_1}.$$
Since $f_{\alpha,h}^\tau:=\sum_{k=1}^{m_2}f_k\psi_k$ with $\bfa=[f_k]_{1\le k\le m_2}$, a choice of $g^\tau=\psi_k$ in the discrete forward problem \eqref{weak.inverse.modified} reduces it into a matrix equation given by
$$\bB\bfa=\bl$$
where
	$$\bB:=\big[(\Phi T_h\psi_k)^*(\Phi T_h\psi_l)+\alpha b(\psi_k,\psi_l)\big]_{1\le k,l\le m_2} \mbox{ and }\bl=\big[m^*(\Phi T_h\psi_l)\big]_{1\le l\le m_2}.$$
Recall the definition of $\Phi(\cdot)$ given in \eqref{measurement}. Let's consider the case where the measurement functionals $u \rightarrow (\phi_i,u ) $ are chosen as the averages of $u$ over the sets $\omega_i \subset \Omega,\, i=1,\cdots,N$. Define
$$\bW:=\left[\frac{\int_{\omega_i}T_h\psi_k}{\omega_i}\right]_{1 \le i\le N, \\1 \le k \le m_2} \mbox{ and } \bC:= \big[b(\psi_k,\psi_l)\big]_{1 \le k,l \le m_2}. $$%\mbox{ Then } \bB=\bW^*\bW+\bC.$$
Then $\bB=\bW^*\bW+\bC.$ Since $T_h\psi_k \in V_h$ for $k=1,\cdots,m_2$, $T_h\psi_k:=\sum_{i=1}^{m_1}\beta_i^k\phi_i$ and the unknowns $\bbeta^k$ can be computed using the forward matrix equation
$$\bA\bbeta^k=\bS(:,k), \mbox{ where } \bS=\left[\int_\O\phi_i\psi_j\right]_{1 \le i \le m_1\\ 1 \le j \le m_2}. $$
For the \BFS finite elements $V_h$ with 16 degree of freedoms in a rectangle, if $\omega_i, i=1,\cdots,N$ is chosen as one of the \BFS rectangle, say $R$, then
$$\int_{\omega_i}T_h \psi_k=\int_{R}\sum_{j=1}^{16}\beta_j^k\phi_{j_{|R}}=\sum_{j=1}^{16}\beta_j^k\int_R\phi_{j_{|R}}.$$
Consequently, the matrix $\bB$ and the vector $\bl$ can be computed and then the matrix formulation $\bB \bfa=\bl$ can be used to evaluate $f_{\alpha,h}^\tau$.

\smallskip \noindent
 The algorithm where the measurement functionals are chosen as the averages of $u$ over the sets $\omega_i \subset \Omega,\, i=1,\cdots,N$ and $b(f,g)=(f,g) \fl f,g\in F$ is described now.

\medskip 
\noindent		\textsc{\underline{Algorithm:}}\noindent\\
\begin{algof}
		\SetKwFunction{fassemble}{Assemble}
\SetKwHangingKw{assemble}{\fassemble}

\DontPrintSemicolon
\For {$\cT_j, j=0,1,2,\cdots$}{
	\aset{the regularization parameter $\alpha$ and the locations $\omega_i, i \le i \le N$}%, $err$=1
	
	\acompute{the local basis functions, its gradients and Hessians of $V_h$}
	
		\acompute{the local basis functions of $F^\tau$}
	
	\assemble{the global matrices $\bA$, $\bS$ and $\bC$ }
	\acompute{the vector $\bbeta^k$ for $1\le k \le m_2$ and hence the matrix $\bW$}
		
		\assemble{the matrices $\bB$ and $\bl$}
	\acompute{the vector $\bfa$}
	end
}
\end{algof}	
%%%%%%%%%%%%%%%%%%%%%%%%%%%%%%%%%%%%%%%%%%%%%%%%%%%%%%%%%%
\subsection{Numerical results for \BFS elements}\label{sec.numericalresultsBFS}
Numerical results for conforming FEM for $k=0,1,2$ are presented in this section to justify theoretical estimates. The \BFS finite elements with 16 degrees of freedom in a rectangle are considered to solve the forward problem as it has lesser number of degree of freedoms in comparison with the Argyris finite elements with 21 degrees of freedom in a triangle (Figure \ref{fig.cfem}). The finite dimensional subspace $F^\tau$ of $F$ is also chosen as the \BFS finite elements, that is $p=q=3$. For simplicity, $h=\tau$. In the refinement process, each rectangle is divided into 4 equal sub-rectangles by connecting the midpoint of the edges. Figure \ref{fig.square} shows the initial triangulation of a square domain and its uniform refinement for \BFS finite element.

\begin{figure}[h]
	\begin{center}
		\begin{minipage}[H]{0.4\linewidth}
			{\includegraphics[width=5cm]{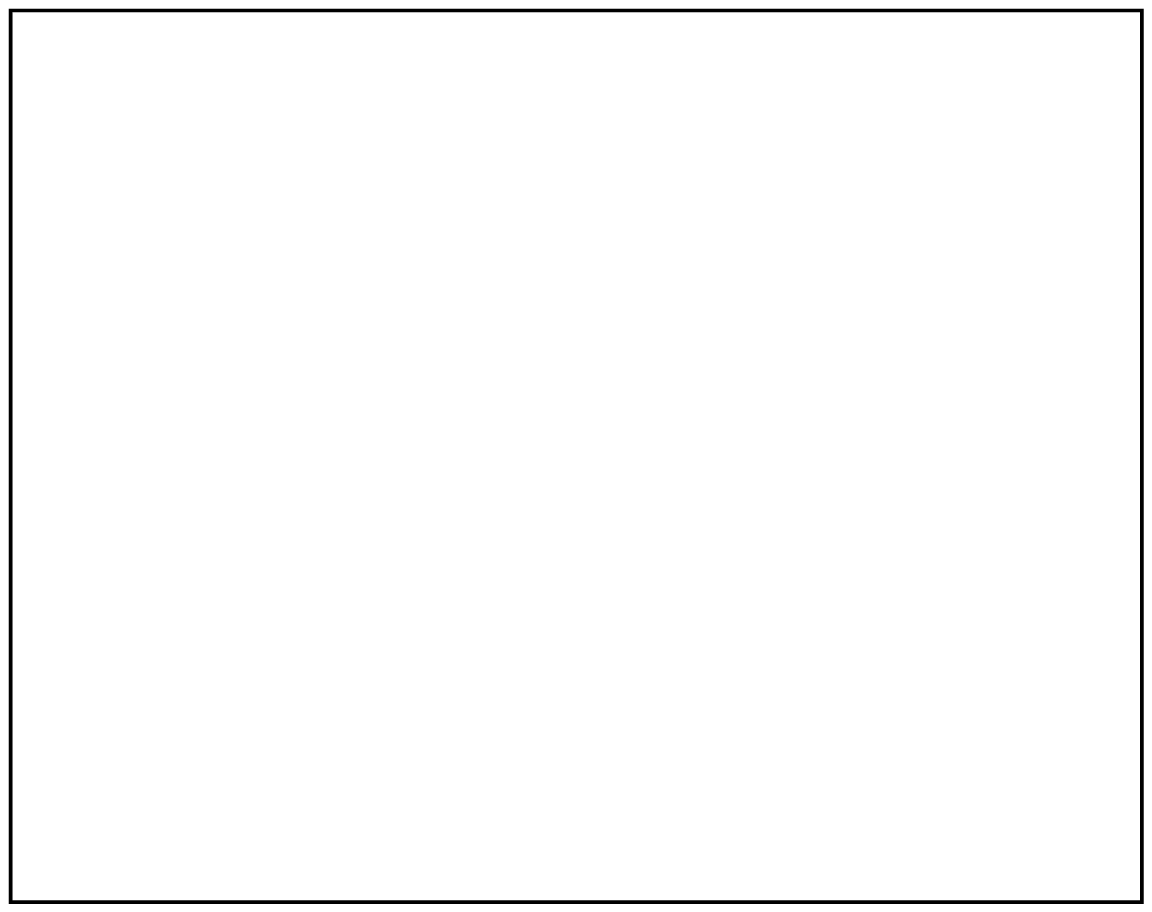}}
		\end{minipage}
		\begin{minipage}[H]{0.4\linewidth}
			{\includegraphics[width=5cm]{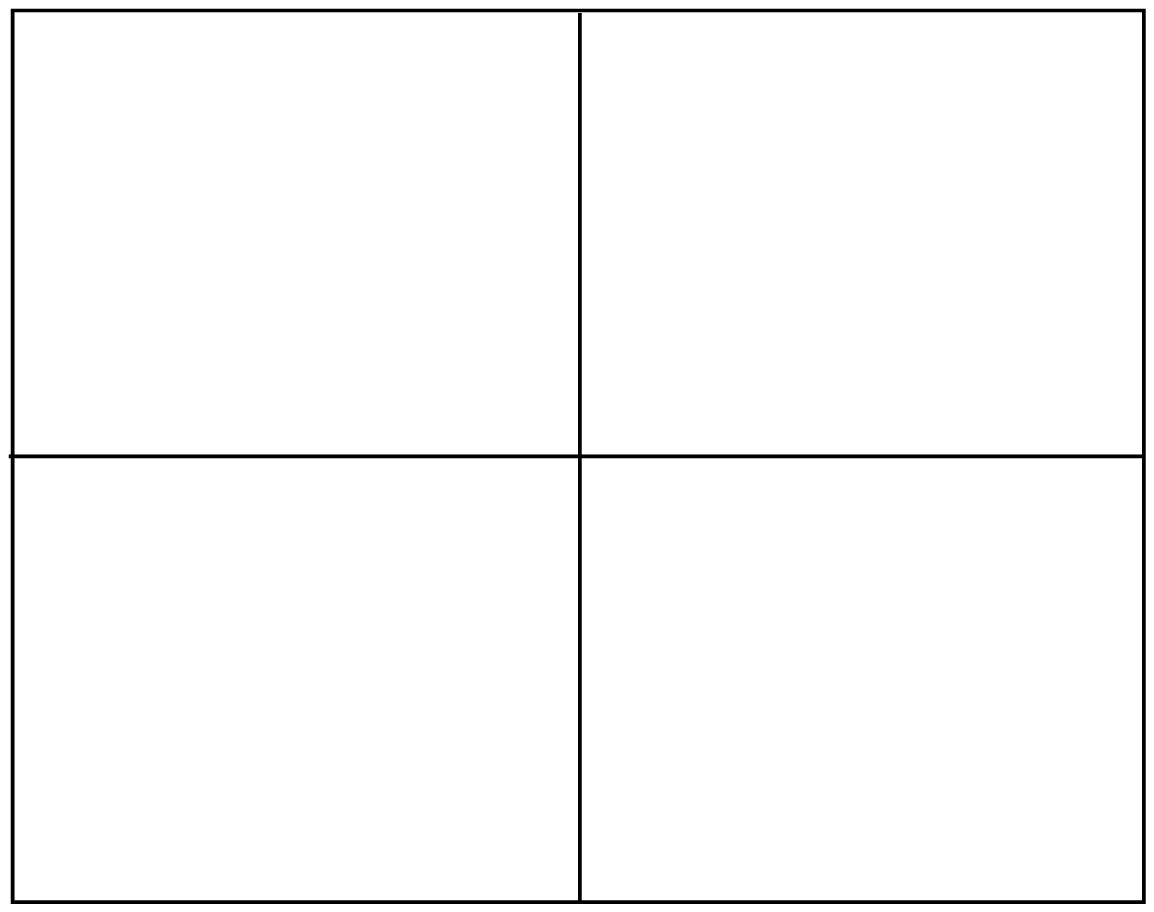}}
		\end{minipage}
		\caption{Initial triangulation and its red-refinement of square domain for \BFS finite element}\label{fig.square}
	\end{center}
\end{figure}

\smallskip 
Let $\ndof$ denotes the number of degrees of freedom of $V_h$ over $\O$. Let $m_{i}$ (resp. $f_{\alpha,i}$) be the discrete solution $m_h$ (resp. $f_{\alpha,h}^h$) of $m$ (resp. $f_\alpha$) at the $i$th level for $i=1,2,\cdots,L$. Define for $f \in F$, a closed subspace of $H^k(\O)$, 
$$ \err_{i}^k(m):=\|m-m_{i}\|=\|\Phi Tf-\Phi T_{i}f\|,$$
where $T_if=u_i$ is the discrete solution of $u \in V$ at the $i$th level. The corresponding order of convergence is computed using the formula 
$$ \order_{i}(m):=\log(\err_{i}^k(m)/\err_{i+1}^k(m))/\log(2).$$
Since $f_\alpha$ is unknown, the $H^k$ error $\| f_\alpha - f_{\alpha,i}\|_k$ is approximated by $\|f_{\alpha,i} -f_{\alpha,L}\|_k$. Define
$$\err_i^k(f_\alpha):=\|f_{\alpha,i} -f_{\alpha,L}\|_k.$$
Note that $\err_L^k(f_\alpha)=0$. The order for $f_\alpha$ is approximated by $$\order_i^k(f_\a):=\log(\err_i^k(f_\a)/\err_{L-1}^k(f_\a))/\log(2^{L-1-i}).$$
If the displacement $u$ is unknown, then in the implementation, $m=\Phi u$ is replaced by $m_i=\Phi u_i$ and $\err_i(m):=\|m_i -m_{L}\|.$
%The order in $L^2$ norm for $m$ is approximated by $$\order_h^k(m)=\log(\err_h^k(m)/\err_{h/2}^k(m))/\log(2).$$ 

\smallskip
	
The results are presented for three particular cases $k=0,1,2$ and for two different domains (Section \ref{sec.squaredomain} and Section \ref{sec.lshapeddomain}) where the functions $\phi_j,\, j=1,\cdots,N$ are defined as
\[\phi_j(x)=\begin{cases}
\frac{1}{|\omega_j|} \quad &\mbox{ if } x \in \omega_j,\,\omega_j \subset \O\\
0 \quad &\mbox{ otherwise.}
\end{cases}\]
This and \eqref{measurement} imply that the $j$th element of the measurement $m$ is the average of $u$ over the set $\omega_j$.

\smallskip

From the characterization of Sobolev spaces using Fourier transform \cite{Kesavan}, an alternative definition of $H^s(\R^2)$ for any given $s \in [0,\infty)$ is given by
$$H^s(\R^2):=\{v \in L^2(\R^2);(1+|\xi|^2)^{s/2}\hat{v} \in L^2(\R^2)\},$$% \quad \fl s \in [0,\infty),$$
where $\hat{v}$ is the Fourier transform of $v$. In a similar way, $H^{-s}(\R^2)$ can be defined. This, the definition of $\phi_j(\cdot)$ and the fact that $\int \frac{1}{(1+|\xi|^2)^t }d\xi$ is finite only when $2t>1$ prove that $\phi_j \in H^{1/2-\beta}(\O), \beta>0$.
\begin{figure}[h!!]
	\begin{minipage}[H]{0.4\linewidth}
		{\includegraphics[width=5cm]{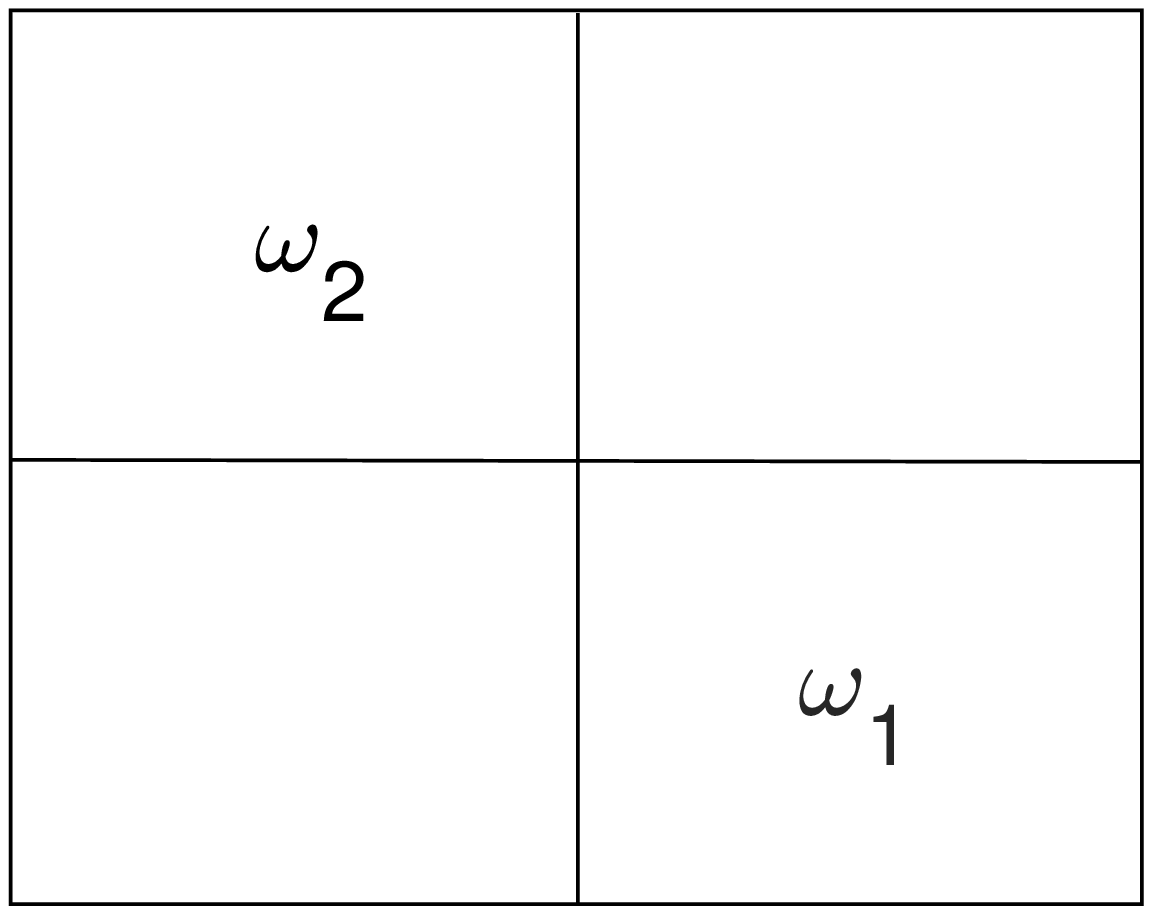}}
	\end{minipage}
	\begin{minipage}[H]{0.4\linewidth}
		{\includegraphics[width=5cm]{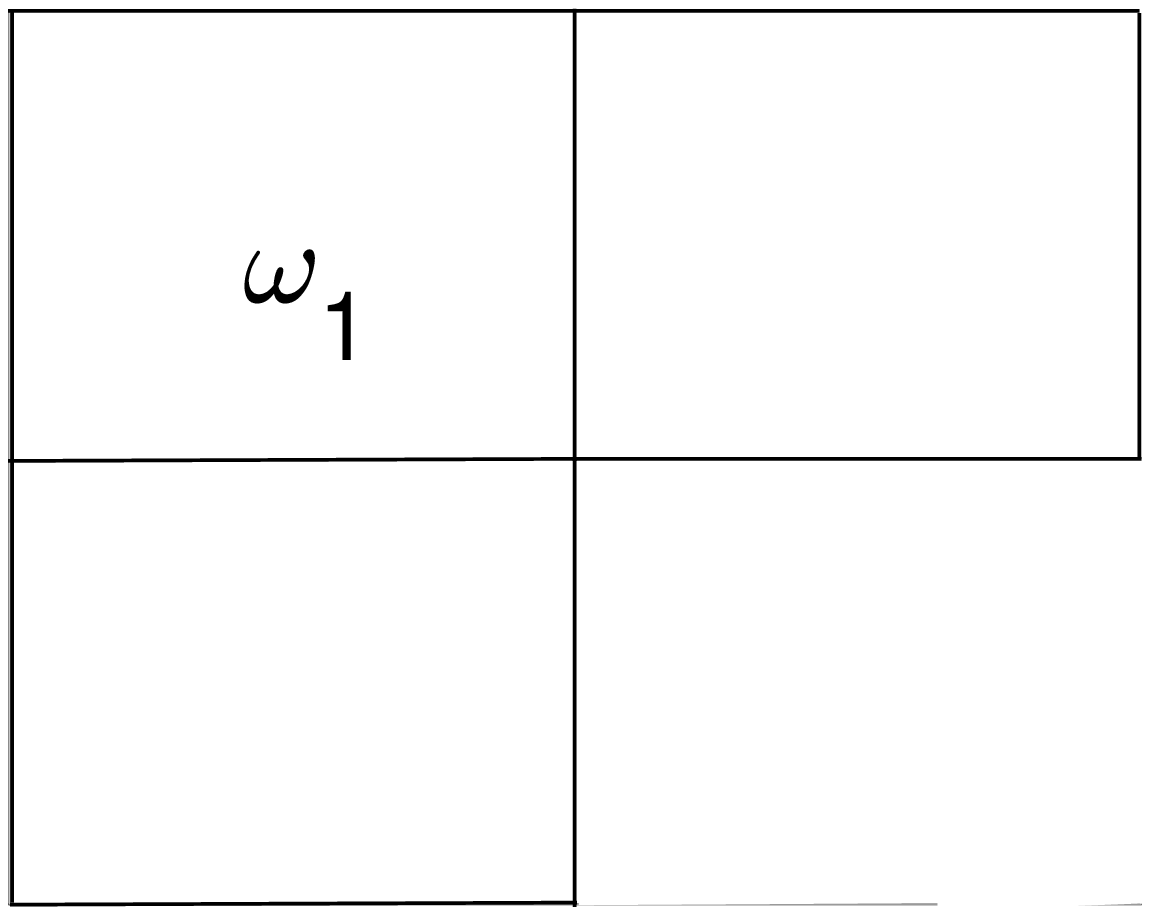}}
	\end{minipage}
	\caption{Location of $\omega_1$ and $\omega_2$ for square domain (left) and location of $\omega_1$ for L-shaped domain (right), \BFS FEM}\label{fig.square.w12.boundary}
\end{figure}
\subsubsection{\bf{Square domain}}\label{sec.squaredomain}
Let the computational domain be $\Omega=(0,1)^2$. The sets $\omega_j,\,j=1,2$ are illustrated in Figure \ref{fig.square.w12.boundary}, left. Since $\phi_j \in H^{1/2-\beta}(\O)$ and $\O$ is convex, the weak formulation of the auxiliary problem \eqref{aux} and Lemma \ref{bih_reg_res} lead to $\xi_j \in H^4(\O)$, that is, $r=4$ in Assumption \ref{assumption}.

\medskip

%$\bullet$ {\bf{Case 1: $L^2$ regularization $(k=0)$.}} 
Consider the case $F=L^2(\O)$ which yields $k=0$. Lemma \ref{bih_reg_res} shows that $u \in H^4(\O)$. Hence, from Assumption \ref{assumption.regularity.forward}, $s=4$.
%Since $p=3$, Theorem \ref{thm.forwarderror} and \eqref{thm.forwarderror.cfem} imply
%$$\|\Phi Tf-\Phi T_hf\|\le Ch^{4}.$$
 The bilinear form $b(f,g)$ is taken as the $L^2$ inner product given by $b(f,g)=(f,g)$ for all $f,g \in L^2(\O)$. It can easily shown that $b(\cdot,\cdot)$ is a coercive continuous linear form on $F$. The combination of this, \eqref{basis_eta}, \eqref{def.falpha.basis} and $\xi_j \in H^4(\O)$ read $f_\a \in  H^4(\O)$, that is, $\ell=4$. 

\smallskip

The conditions $q\ge\ell-1$ and $p \ge \max\{r,s\}-1$ are satisfied. As a consequence of Theorems \ref{thm.forwarderror} and \ref{thm.totalerror} for conforming FEMs (see \eqref{thm.forwarderror.cfem} and \eqref{thm.totalerror.cfem} for details),
\begin{equation}\label{bfs.squarek0}
\|m-m_h\|= \mathcal{O}(h^{4}) \mbox{ and }\|f_\alpha-f_{\alpha,h}^h \| = \mathcal{O}(h^{4}).
\end{equation}
Two examples are considered in this case: the displacement $u$ is known (Example 1) and $u$ is unknown (Example 2). The number of unknowns/degrees of freedom $\ndof$ at the last level ($h=0.0221,\, L=6$) is 15876.

\smallskip

$\bullet$ {\bf{ Example 1: }}The model problem is constructed in such a way that the displacement $u$ is known. The data is chosen as $u=x^2 y^2 (1-x)^2 (1-y)^2 \in V$. The source term $f \in L^2(\O)$ is then computed using $f=\Delta^2u$.

\medskip

 The numerical results are presented in Table \ref{table.squarem} for $\alpha=10^{-3}$ and $\alpha=10^{-7}.$ It can be seen that the rate of convergence is quartic for $m$ in $L^2$ norm. Also, the order of convergence for the reconstructed regularised approximation of the source field $f_\alpha$ in $L^2$ norm is close to $\mathcal{O}(h^4)$. The theoretical rates of convergence given in \eqref{bfs.squarek0} are confirmed by these numerical outputs. We observe in this example that as $\alpha$ tends to $0$, the order of $f_\alpha$ is getting close to $4$ faster.

\begin{table}[h!!] 
	\caption{\small{Convergence results for Example 1, Square domain, \BFS, $k=0$, $\alpha=10^{-3}$ and $10^{-7}$}}
	{\small{\scriptsize
			\begin{center}
				\begin{tabular}{ ||c|||c|c||c|c||c|c||c|c||}
					\hline%\\[-5.5pt]  &&\\[-10pt]
					$i$&$h$ &\ndof&$\err_i^0(m)$&\order &$\err_{i}^0(f_{10^{-3}})$ & \order &$\err_{i}^0(f_{10^{-7}})$ &\order \\	
					\hline%\\[-10pt]  &&\\[-10pt]
1&0.7071&4&  0.000143&  4.0171& 0.000062& 4.1398&0.362101&4.0393\\
2&0.3536& 36& 8.86e-6& 4.0034& 3.09e-6& 4.0781&0.021149&4.0198\\
3&0.1768&196& 5.52e-7&4.0003&1.10e-7& 3.7088& 0.001296& 4.0158\\
4&0.0884&900& 3.45e-8&4.0000&9.20e-9& 3.8409&0.000088& 4.0634\\
5&0.0442&3844&2.16e-9&4.0010&6.4e-10&    -& 0.000005&- \\
%0.000002209708691   1.587600000000000   0.000000001230439                   0   0.000000000000013                   0                   0                   0
					\hline	
				\end{tabular}
			\end{center}
	}}\label{table.squarem}
\end{table}

$\bullet$ {\bf{ Example 2: }}In this example, the source term $f$ is chosen as $\exp(x+y)$ with unknown displacement $u$ and $\alpha=10^{-5}$. Table \ref{table.squaremh} shows the errors and orders of convergence for $m$ and $f_\alpha$. The numerical results are similar to those obtained for Example 1 in Table \ref{table.squarem}.

\begin{table}[h!!] 
	\caption{\small{Convergence results for Example 2, Square domain, \BFS, $u$ unknown, $k=0$, $\alpha=10^{-5}$}}
	{\small{\scriptsize
			\begin{center}
				\begin{tabular}{ ||c||c||c||c|c||c|c||}
					\hline
				$i$&	$h$ &\ndof &$\err_i^0(m)$&\order  &$\err_{i}^0(f_{10^{-5}})$ & \order \\	
					\hline%\\[-10pt]  &&\\[-10pt]
1&0.7071&4&0.000231&3.8309& 0.010066& 3.6198\\
2&0.3536& 36&  0.000021&3.9557&0.001435&3.8897\\
3&0.1768&196&0.000001& 3.9999& 0.000106& 3.9582\\
4&0.0884&900&9.36e-8&4.0544& 0.000007&4.0296\\
5&0.0442&3844&5.64e-9&-& 4.41e-7& -\\						
					\hline	
				\end{tabular}
			\end{center}		
	}}\label{table.squaremh}
\end{table}
%The source function $f$ and the discrete reconstructed regularised approximation of the source field $f_{\alpha,h}^h$ for $\alpha=10^{-5}$ for $k=0$ are depicted in Figure \ref{fig.squaremh}.

\subsubsection{\bf{L-shaped domain}}\label{sec.lshapeddomain}
Consider the non-convex L-shaped domain given by $\Omega=(-1,1)^2 \setminus\big{(}[0,1)\times(-1,0]\big{)}$. This example is particularly interesting since the solution is less regular due to the corner singularity.

The location $\omega_1\, (N=1)$ is depicted in Figure \ref{fig.square.w12.boundary}, right. Since $\O$ is non-convex, Lemma \ref{bih_reg_res} shows $u, \xi  \in H^{2+\gamma}(\O),\,\gamma\in(\half,1)$. The number of unknowns \ndof at the last level ($h=0.0221,\,L=6$) is  48132.

%\begin{figure}[h!]
%	\begin{center}
%		{\includegraphics[width=7cm]{Lshaped1.eps}}
%		\caption{L-shaped domain, location of $\omega_1$ }\label{fig.Lshaped.w1.boundary}
%	\end{center}
%\end{figure}

\medskip

$\bullet$ {\bf{Case 1: $L^2$ regularization $(k=0)$.}} The space $F$ and the bilinear form $b(\cdot,\cdot)$ on $F$ are chosen as in Section \ref{sec.squaredomain}.  The source term $f$ is chosen such that the forward problem has the exact singular solution {{\cite[Section 3.4.1]{Grisvard92}}} given by
$
u=({R}^2 \cos^2\theta-1)^2 ({R}^2 \sin^2\theta-1)^2 {R}^{1+ \gamma}g_{\gamma,\omega}(\theta)
$
where $ \gamma\approx 0.5444837367$ is a non-characteristic 
root of $\sin^2( \gamma\omega) =  \gamma^2\sin^2(\omega)$, $\omega=\frac{3\pi}{2}$, and
$g_{\gamma,\omega}(\theta)=(\frac{1}{\gamma-1}\sin ((\gamma-1)\omega)-\frac{1}{ \gamma+1}\sin(( \gamma+1)\omega))(\cos(( \gamma-1)\theta)-\cos(( \gamma+1)\theta))$ 
$-(\frac{1}{\gamma-1}\sin(( \gamma-1)\theta)-\frac{1}{ \gamma+1}\sin(( \gamma+1)\theta))
(\cos(( \gamma-1)\omega)-\cos(( \gamma+1)\omega))$ and $\left(R,\theta\right)$ are the polar coordinates.

\smallskip

The definition of $b(\cdot,\cdot)$, \eqref{basis_eta}, \eqref{def.falpha.basis} along with $\xi  \in H^{2+\gamma}(\O)$ imply $f_\a \in H^{2+\gamma}(\O)$. The conditions $q\ge \ell-1$ and  $p \ge \max\{r,s\}-1$ are trivially satisfied since $s=r=\ell =2+\gamma$. Hence, from Theorems \ref{thm.forwarderror} and \ref{thm.totalerror} together with \eqref{thm.forwarderror.cfem} and \eqref{thm.totalerror.cfemvc},
$$\|m-m_h\| = \mathcal{O}(h^{2\gamma}) \mbox{ and }\|f_\alpha-f_{\alpha,h}^h \| = \mathcal{O}(h^{2\gamma}).$$

\smallskip

The results of this numerical experiment ($\err_i^0(m)$ and $\err_{i}^0(f_{10^{-5}})$) are displayed in Table \ref{table.Lshapedmh}. As seen in the table, we obtain $\mathcal{O}(h^{2\gamma})$ for $m$ in $L^2$ norm. This numerical order of convergence clearly match the expected order of convergence, given the regularity property of the solution. However, superconvergence is obtained for $f_\a$ in $L^2$ norm which indicates that the numerical performance is carried out in the non-asymptotic region for this domain that has a corner singularity combined with the fact that the exact solution $f_\a$ is unknown. Similar observations can be found in literature, for example \cite[Section 6.3]{carstensen2020morley}.

\medskip

$\bullet$ {\bf{Case 2: $H^1$ regularization $(k=1)$.}} In this example, $F=H^1_0(\O)$. The bilinear form $b(f,g)$ is defined by  $b(f,g)=\int_\O \nabla f\nabla g \dx$ for all $f,g \in F$. Then, $b(\cdot,\cdot)$ is a coercive continuous linear form on $F$. Here, we report the results of numerical test for the source field $f \in H^1_0(\O)$ given by
\begin{align*}
&f(R,\theta)=\left( R^2\cos^2\theta-1\right) \left(R^2\sin^2\theta-1\right)R^{2/3}\left(1-\cos\theta\right) \left(1+\sin\theta\right).
\end{align*}
From \eqref{basis_eta} and \cite[Remark 2.4.6]{Grisvard92}, we have $\eta_1 \in H^{1+\rho}(\O),\,0< \rho<2/3$ and hence \eqref{def.falpha.basis} implies $f_\a \in  H^{1+\rho}(\O)$, that is, $\ell=1+\rho$. Also, $q\ge\ell-1$  and $p \ge \max\{r,s\}-1$. As a result, Theorems \ref{thm.forwarderror} and \ref{thm.totalerror} (\eqref{thm.forwarderror.cfem} and \eqref{thm.totalerror.cfemvc}) lead to
$$\|m-m_h\|= \mathcal{O}(h^{2\gamma}) \mbox{ and }\|f_\alpha-f_{\alpha,h}^h \|_1 =\mathcal{O}(h^{\rho}).$$ 
However, the numerical results tabulated in Table \ref{table.Lshapedmh} ( $\err_h^1(m)$ and $\err_{h}^1(f_{10^{-5}})$) show super convergence for $m$ and $f_\alpha$. This can be justified in a similar way as in Case 1.
\begin{table}[h!!] 
	\caption{\small{Convergence results, L-shaped domain, \BFS, $\alpha=10^{-5},$ $k=0,1,2$}}
	{\small{\scriptsize
			\begin{center}
				\begin{tabular}{ ||c||c||c||c|c||c|c||c||c||}
					\hline
				$i$&	$h$ &\ndof &$\err_i^0(m)$&\order  &$\err_{i}^0(f_\alpha)$ & \order&$\err_i^1(m)$  \\	
					\hline%\\[-10pt]  &&\\[-10pt]
				1&	0.7071&20&   0.042987& 1.9235& 1.538154& 1.0308&0.000217\\
				2&0.3536&   132& 0.011332& 1.2661& 1.396468& 1.3280&0.000068\\
				3&0.1768&644&0.004712& 1.1194&0.666975& 1.4589&0.000028\\
				4&0.0884& 2820&  0.002169& 1.0946&0.275330&1.6414&0.000011\\
				5&0.0442& 11780&  0.001016&1.0901&0.088254&-&0.000004\\   		
					\hline	
				\end{tabular}
			\end{center}
			\begin{center}
				\begin{tabular}{  ||c||c||c|c||c|c||c|c||c||}
					\hline
					$i$&\order &$\err_{i}^1(f_\alpha)$ & \order &$\err_i^2(m)$&\order  &$\err_{i}^2(f_\alpha)$ & \order  \\	
					\hline%\\[-10pt]  &&\\[-10pt]
		1&1.4772& 0.492943&2.0819&0.000113&1.4611&0.002226&0.9691\\
		2&1.4091&0.120080&2.0967&0.000036&1.4035& 0.001039& 0.9256\\
		3&1.4747&0.028848& 2.1163&0.000015&1.4734&0.000563& 0.9464\\
		4&1.6446&0.006839& 2.1560&0.000006& 1.6443&0.000311&1.0390\\
		5&-& 0.001534&-&0.000002&-&0.000152&-  \\
					\hline	
				\end{tabular}
			\end{center} 			
	}}\label{table.Lshapedmh}
\end{table}
\medskip

$\bullet$ {\bf{Case 3: $H^2$ regularization $(k=2)$.}} Choose $F=V$ and $b(f,g)=\int_\O \Delta f\Delta g \dx=a(f,g)$ for all $f,g \in V$. The source field $f \in V$ is chosen as the displacement $u$ given in Case 1. Arguments analogous as in Case 1 leads to $\ell=2+\gamma.$ Consequently, Theorems \ref{thm.forwarderror} and \ref{thm.totalerror} read
$$\|m-m_h\|=\mathcal{O}(h^{2\gamma}) \mbox{ and }\|f_\alpha-f_{\alpha,h}^h \|_2 =\mathcal{O}(h^{\gamma}),$$ 
which are discussed in \eqref{thm.forwarderror.cfem} and \eqref{thm.totalerror.cfemvc}. The errors and orders of convergence are presented in Table \ref{table.Lshapedmh} ($\err_h^2(m)$ and $\err_{h}^2(f_{10^{-5}})$). The same comments as in Case 2 can be made about this example. Observe that the numerical order of convergence for $m$ is close to $1.6$ for both $k=1$ and $2$ with unknown displacement $u$.% both theoretically ($=2\gamma$) and numerically ($\approx 1.64$).

\subsection{Numerical results for Morley elements}\label{sec.numericalresultsMorley}
This section deals with the results of numerical experiments for the Morley nonconforming FEM for $k=0$ and $k=2$. Figure \ref{fig.squareMorley} shows the initial triangulation
of a square domain and its uniform refinement for the Morley FEM. The measurement $m$ is defined as in Section \ref{sec.numericalresultsBFS} with the locations $\omega_j,\,j =1,2$ depicted in Figure \ref{fig.square.Lshaped.location_Morley} for square and L-shaped domains. The  space $F$ and the bilinear form $b(\cdot,\cdot)$ are chosen the same as that in Section \ref{sec.numericalresultsBFS}. The finite dimensional spaces are $V_h=F^\tau=V_\M$ with $h=\tau$. This implies $p=q=2$.
\begin{figure}[h]
	\begin{center}
		\begin{minipage}[H]{0.4\linewidth}
			{\includegraphics[width=5cm]{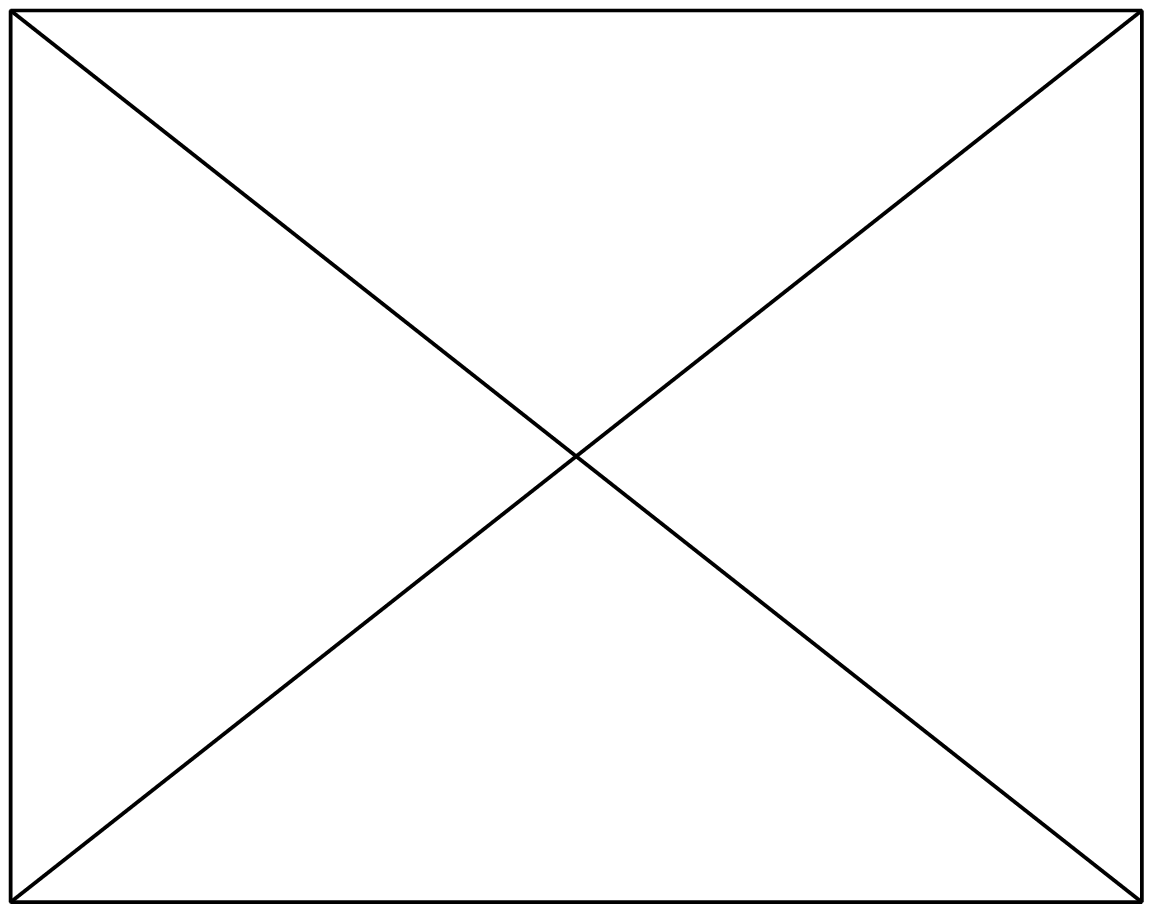}}
		\end{minipage}
		\begin{minipage}[H]{0.4\linewidth}
			{\includegraphics[width=5cm]{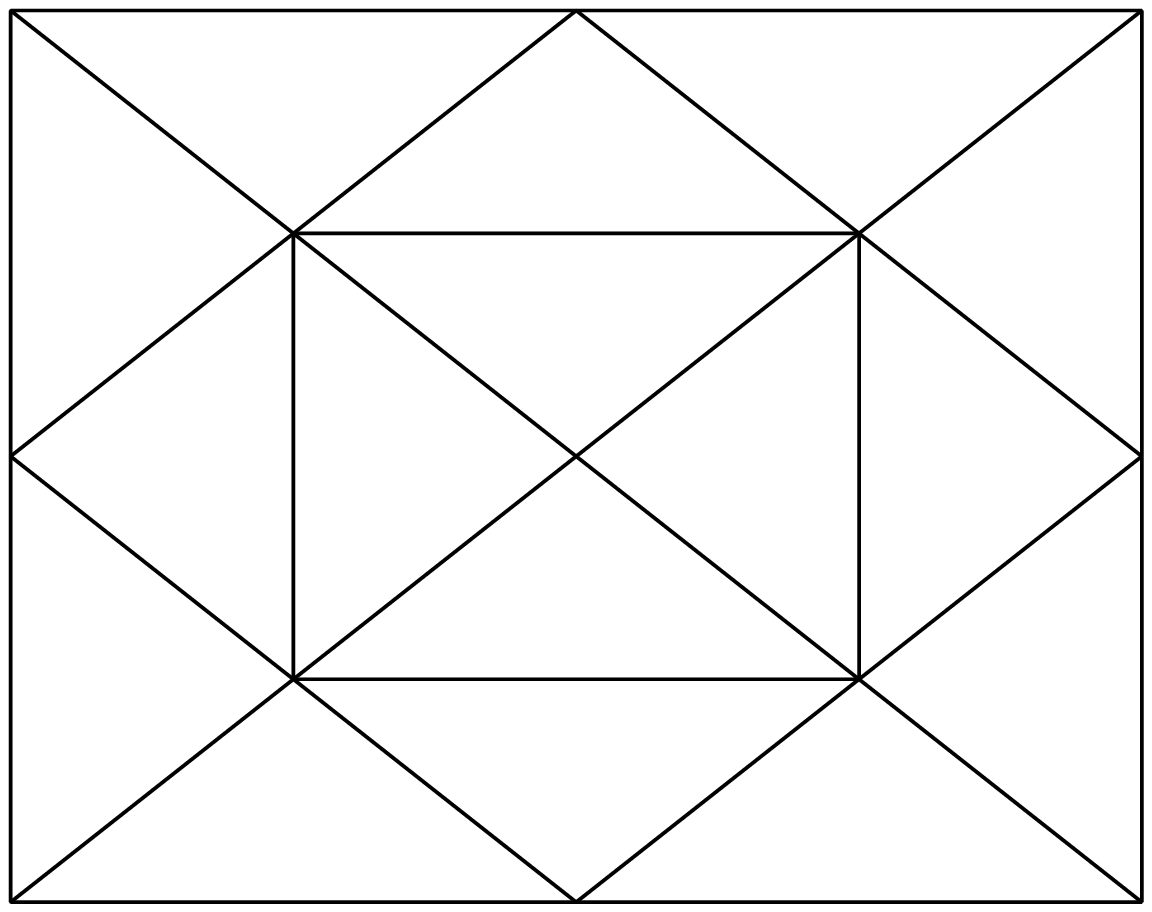}}
		\end{minipage}
		\caption{Initial triangulation and its red-refinement of square domain for Morley FEM}\label{fig.squareMorley}
	\end{center}
\end{figure}
\begin{figure}[h!!]
	\begin{center}
	\begin{minipage}[H]{0.4\linewidth}
		{\includegraphics[width=5cm]{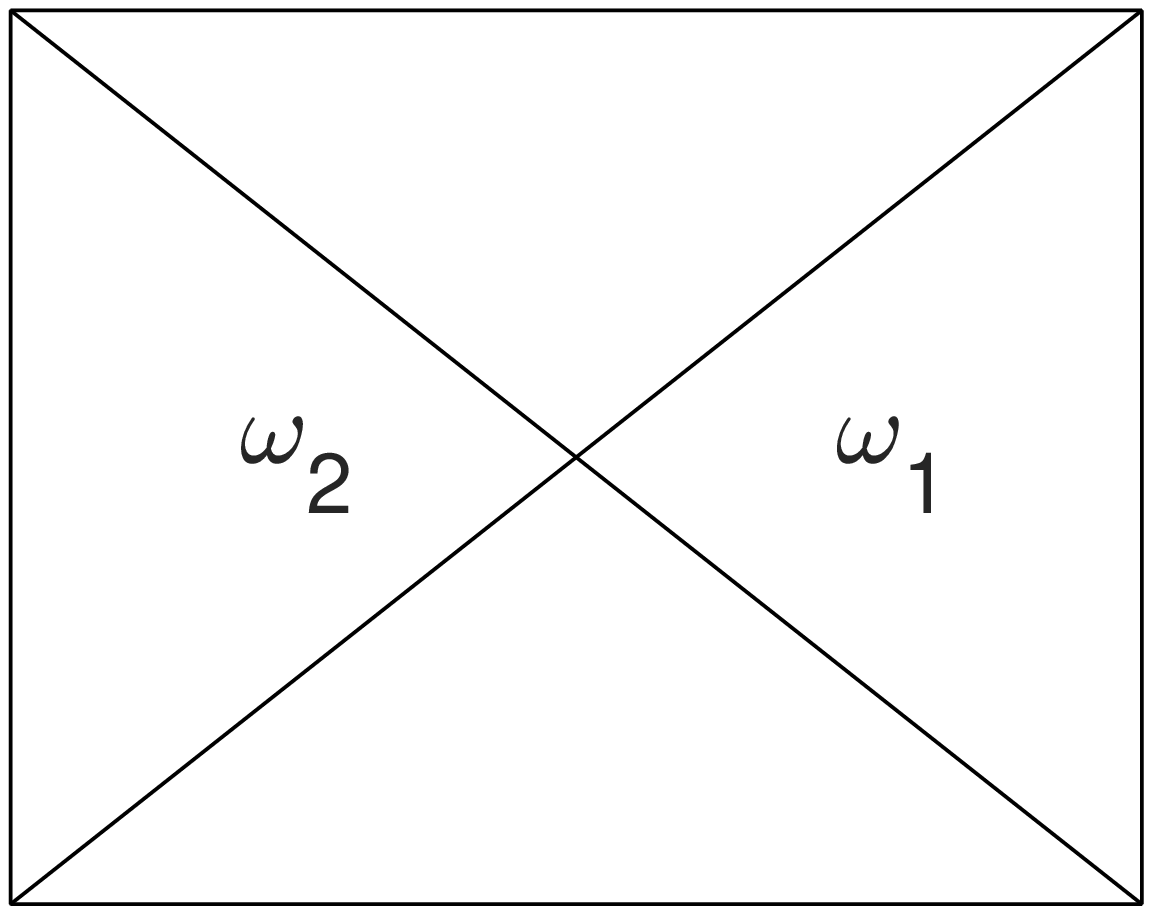}}
	\end{minipage}
	\begin{minipage}[H]{0.4\linewidth}
		{\includegraphics[width=5cm]{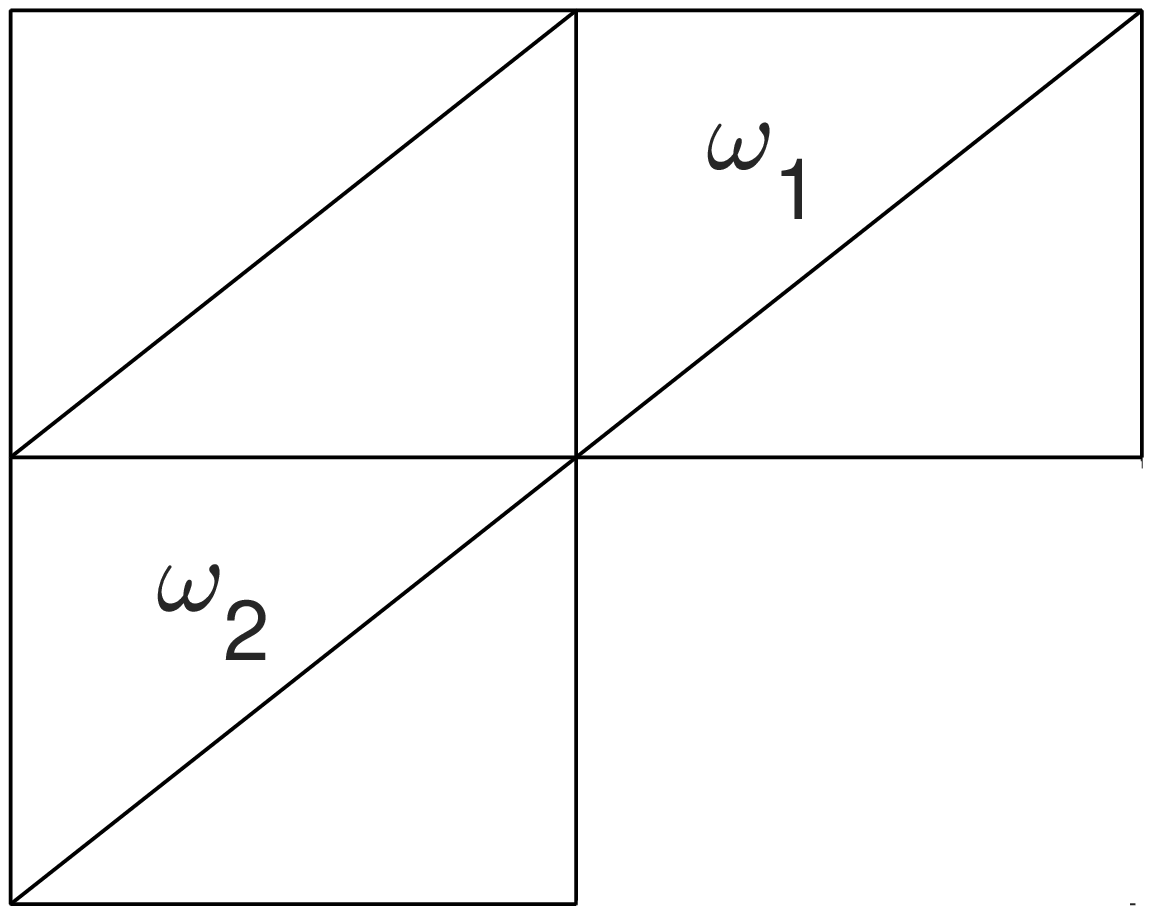}}
		%		\caption{L-shaped domain, location of $\omega_1$ }\label{fig.Lshaped.w1.boundary}
	\end{minipage}
	\caption{Location of $\omega_1$ and $\omega_2$ for square domain (left) and L-shaped domain (right), Morley FEM}\label{fig.square.Lshaped.location_Morley}
\end{center}
\end{figure}
\subsubsection{Square domain}
Consider Example 2 in Section \ref{sec.squaredomain}. The criss-cross mesh with $h=1$ is taken as the initial triangulation of $\O$. The errors and orders of convergence for $m$ and $f_\alpha$ are displayed in Table \ref{table.squaremhMorley}. The number of degrees of freedom \ndof at the last level ($h=0.0156,L=7$) is 32513.

\smallskip

We have $k=0$, $p=q=2$ and $r=s=\ell=4$. Since $F^\tau=V_\M \subset F$, Theorems \ref{thm.forwarderror} and \ref{thm.totalerror} (\eqref{thm.forwarderror.ncfem} and \eqref{thm.totalerror.cfemvc}) result in
$$\|m-m_h\| = \mathcal{O}(h^{2}) \mbox{ and }\|f_\alpha-f_{\alpha,h}^h \| = \mathcal{O}(h^{2}).$$
These orders are confirmed by the numerical results tabulated in Table \ref{table.squaremhMorley}.
\begin{table}[h!!] 
	\caption{\small{Convergence results, Square domain, Morley, $k=0$, $\alpha=10^{-5}$}}
	{\small{\scriptsize
			\begin{center}
				\begin{tabular}{ ||c||c||c||c|c||c|c||}
					\hline
					$i$&	$h$ &\ndof &$\err_i^0(m)$&\order  &$\err_{i}^0(f_{\alpha})$ & \order \\	
					\hline%\\[-10pt]  &&\\[-10pt]
					1&1.0000&5&  0.030649& 2.1958&2.2889&1.9432\\
					2&0.5000&25& 0.004727&2.0705&0.874200&2.0819\\
					3&0.2500&113& 0.001219&2.1091&0.196650&2.0584\\
					4&0.1250&481& 0.000313&2.1836&0.044512&2.0159\\
					5&0.0625&1985& 0.000076&2.3164& 0.010746&1.9815\\						
					6&0.0313&8065& 0.000015&-& 0.002721&-\\
					\hline	
				\end{tabular}
			\end{center}		
	}}\label{table.squaremhMorley}
\end{table}
\subsubsection{L-shaped domain}
In this section, we take the examples in Section \ref{sec.lshapeddomain} for the cases $k=0,2$. Since $\O$ is non-convex, $s=r=2+\gamma$. The number of unknowns \ndof at the last level ($h=0.0221, L=7$) is 48641.
 \medskip
 
$\bullet$ {\bf{$L^2$ regularization $(k=0)$.}} As in Section \ref{sec.lshapeddomain}, $\ell=2+\gamma$. Since $F^\tau \subset F$, Theorems \ref{thm.forwarderror} and \ref{thm.totalerror} together with \eqref{thm.forwarderror.ncfem} and \eqref{thm.totalerror.cfemvc} read
$$\|m-m_h\| = \mathcal{O}(h^{2\gamma}) \mbox{ and }\|f_\alpha-f_{\alpha,h}^h \| = \mathcal{O}(h^{2\gamma}).$$

%\smallskip
%
%$\bullet$ {\bf{Case 2: $H^1$ regularization $(k=1)$.}} The same arguments as in Section \ref{sec.lshapeddomain} for $k=2$ imply $\ell=1+\rho$. The combination of Assumption \ref{assumption.companion} and  Lemma \ref{hctenrich} proves $\delta_6=h^2,\, \delta_7=1$ and $\delta_8=0$. Therefore, Theorems \ref{thm.forwarderror} and \ref{thm.totalerror} lead to
%$$\|m-m_h\|= \mathcal{O}(h^{2\gamma}) \mbox{ and }\|f_\alpha-f_{\alpha,h}^h \|_1 =\mathcal{O}(h^{\rho}).$$ 

\smallskip

$\bullet$ {\bf{$H^2$ regularization $(k=2)$.}} Here, $\ell=2+\gamma$. Consequently, Theorems \ref{thm.forwarderror} and \ref{thm.totalerror} with \eqref{thm.forwarderror.ncfem} and \eqref{thm.totalerror.ncfemvm} show that
$$\|m-m_h\|=\mathcal{O}(h^{2\gamma}) \mbox{ and }\|f_\alpha-f_{\alpha,h}^h \|_2 =\mathcal{O}(h^{\gamma}).$$ 

The errors and orders of convergence for $m$ and $f_\a$ are presented in Table \ref{table.LshapedmhMorley} for $k=0,2.$ The numerical orders of convergence are better than the orders of convergences from the theoretical analysis. A similar observation was made in Section \ref{sec.lshapeddomain} for \BFS FEM.
\begin{table}[h!!] 
	\caption{\small{Convergence results, L-shaped domain, Morley, $\alpha=10^{-5}$, $k=0,2$}}
	{\small{\scriptsize
			\begin{center}
				\begin{tabular}{ ||c||c||c||c|c||c|c||c|c||c|c||}
					\hline
					$i$&	$h$ &\ndof &$\err_i^0(m)$&\order  &$\err_{i}^0(f_\alpha)$ & \order&$\err_i^2(m)$ &\order  &$\err_{i}^2(f_\alpha)$ & \order \\	
					\hline%\\[-10pt]  &&\\[-10pt]
					1&1.4142&5&0.500572&-0.4075&26.104351&1.0631&0.002711&1.6266&0.195481&1.8019\\
					2&	0.7071&33& 0.663949&1.2960&12.021589&1.0493&0.001534&1.8277&0.084437&1.9496\\
					3&0.3536&161& 0.270395&1.7765&12.100613&1.4022& 0.000544&1.9386& 0.015728&1.7913\\
					4&0.1768&705&0.078925&1.8469&5.374215&1.5178&0.000155&2.0022&0.003563&1.6158\\
					5&0.0884&2945&0.021940&1.7881&1.824836&1.4772&0.000042&2.1136&0.00117&1.5579\\
					6&0.0442&12033&0.006353&1.6706&0.655437&- &0.000010&- &0.000379&- \\   		
					\hline	
				\end{tabular}
			\end{center}
%			\begin{center}
%				\begin{tabular}{  ||c||c||c|c||c|c||c|c||c||}
%					\hline
%					$i$&\order &$\err_{i}^1(f_\alpha)$ & \order &$\err_i^2(m)$&\order  &$\err_{i}^2(f_\alpha)$ & \order  \\	
%					\hline%\\[-10pt]  &&\\[-10pt]
%					1&1.6289&0.277398&1.2418&\\
%					2&1.8121&0.368528&1.6547&\\
%					3&1.9320&0.113162&1.6384&\\
%					4&1.9955& 0.030418&1.5099&\\
%					5&2.1064&0.010296&1.4570&\\
%					6&-&0.003750&-&\\
%					\hline	
%				\end{tabular}
%			\end{center} 			
	}}\label{table.LshapedmhMorley}
\end{table}
\medskip

The discrete reconstructed regularised approximation of the source field $f_{10^{-5},h}^h$ using \BFS and Morley FEMs for $k=0$ (resp. $k=2$) performed in square (resp. L-shaped) domain with unknown $u$ are depicted in Figure \ref{fig.squareLshapedmh}.
\begin{figure}[h]
	\begin{center}
		\begin{minipage}[H]{0.4\linewidth}
			{\includegraphics[width=5cm]{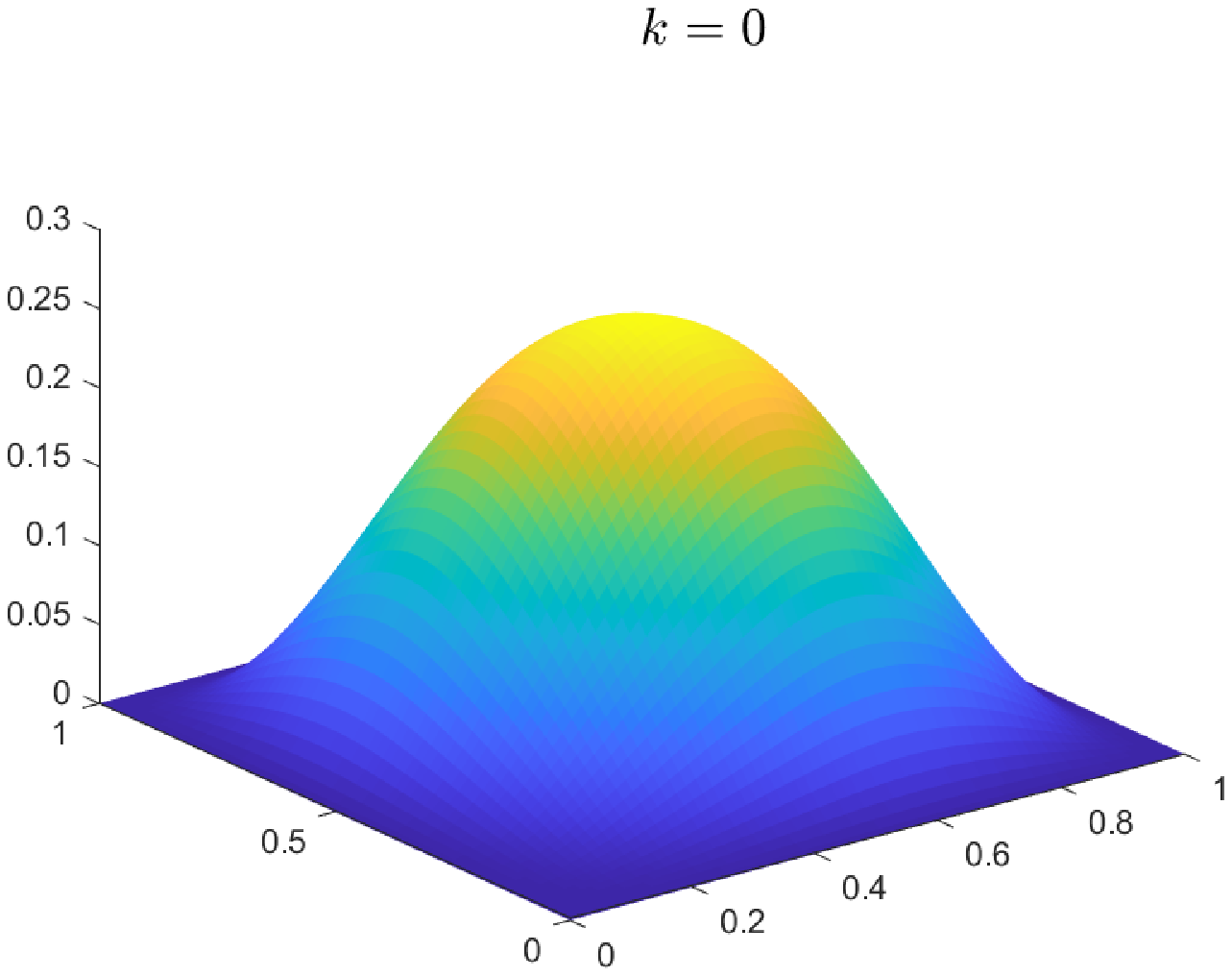}}
		\end{minipage}
		\begin{minipage}[H]{0.4\linewidth}
			{\includegraphics[width=5cm]{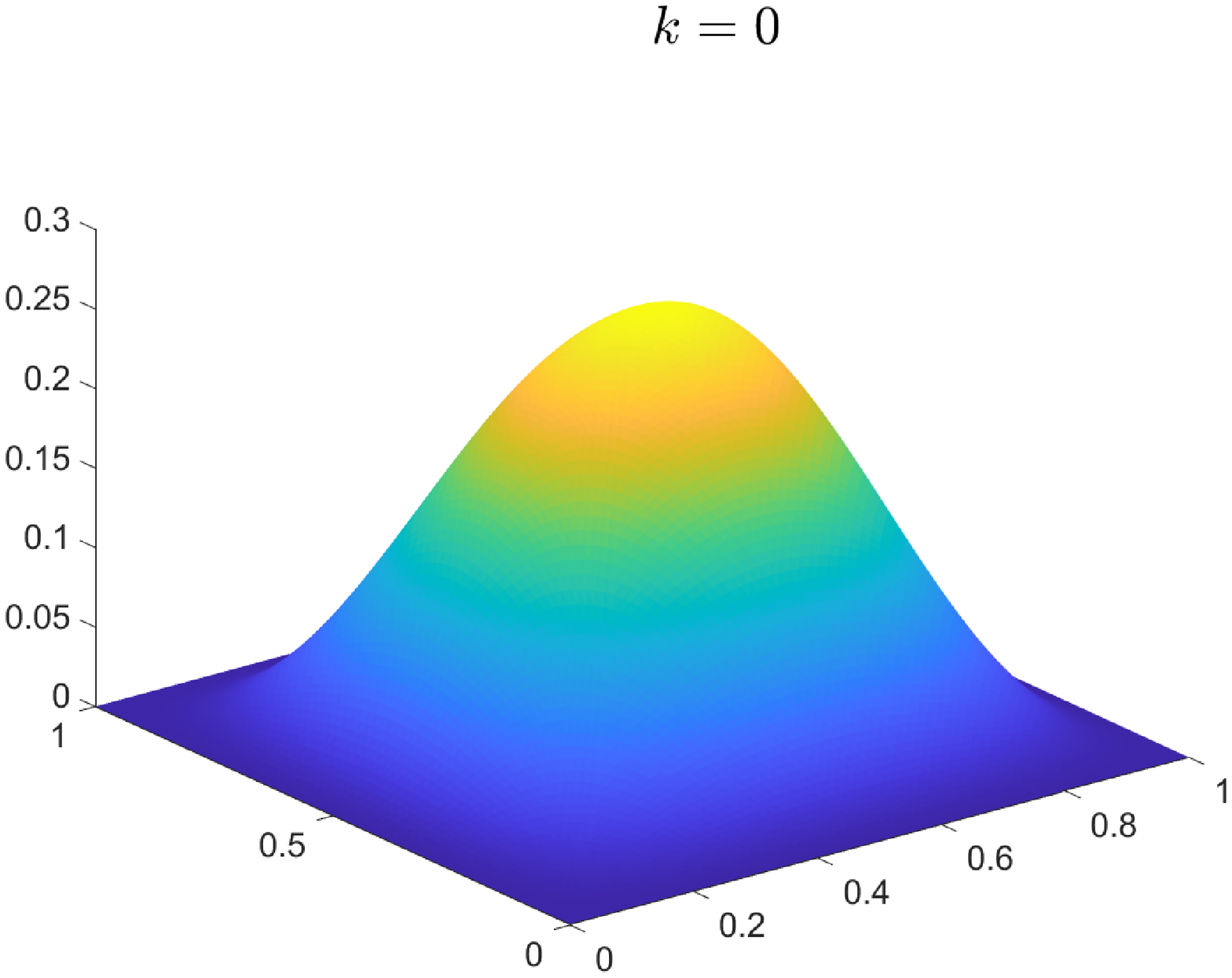}}
		\end{minipage}
%		\begin{minipage}[H]{0.4\linewidth}
%			{\includegraphics[width=5cm]{Figures/fapprox_Lshaped_reg1ue0_BFS.eps}}
%		\end{minipage}
%		\begin{minipage}[H]{0.4\linewidth}
%			{\includegraphics[width=5cm]{Figures/fapprox_Lshaped_reg1ue0_Morley.eps}}
%		\end{minipage}
		\begin{minipage}[H]{0.4\linewidth}
			{\includegraphics[width=5cm]{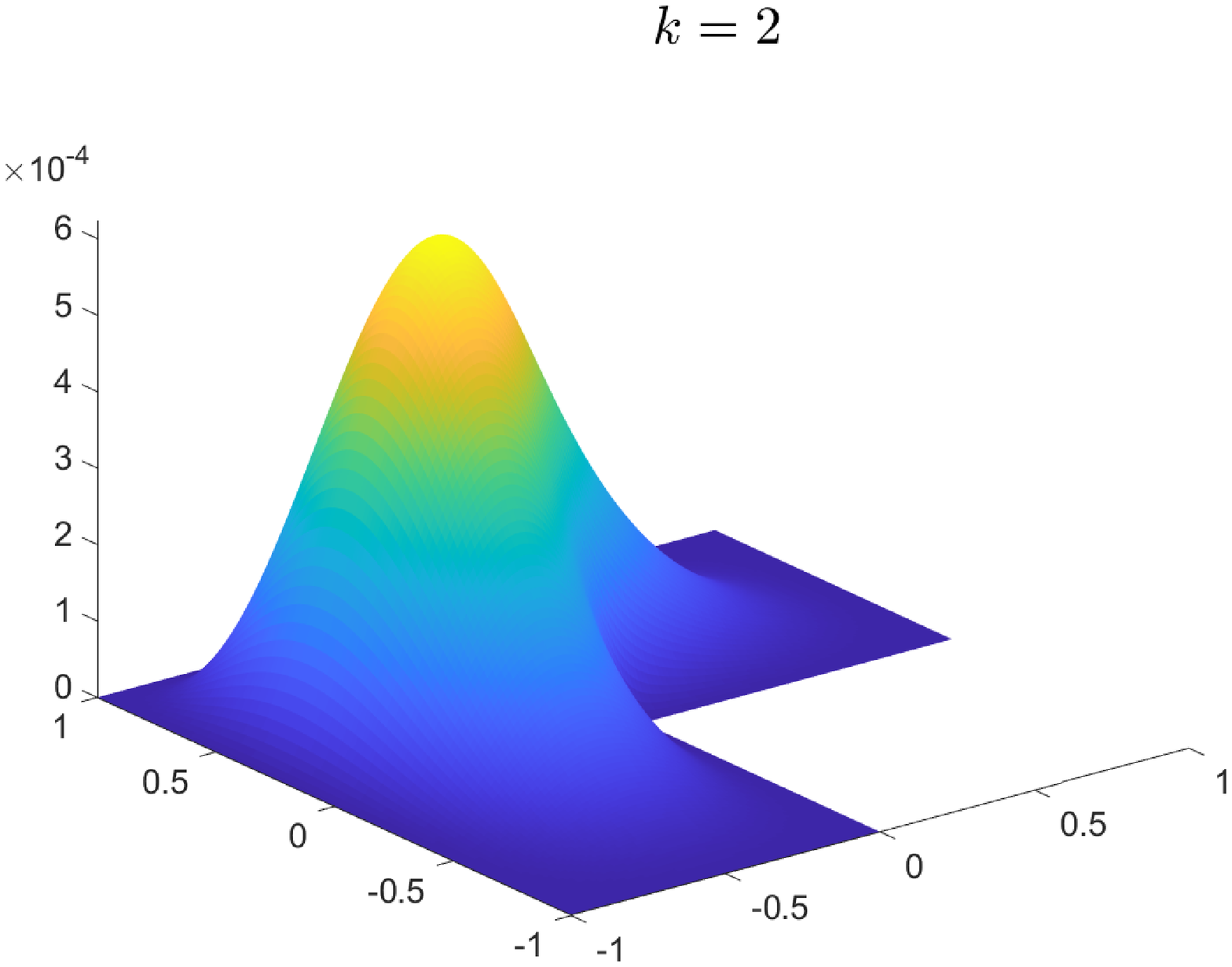}}
		\end{minipage}
		\begin{minipage}[H]{0.4\linewidth}
			{\includegraphics[width=5cm]{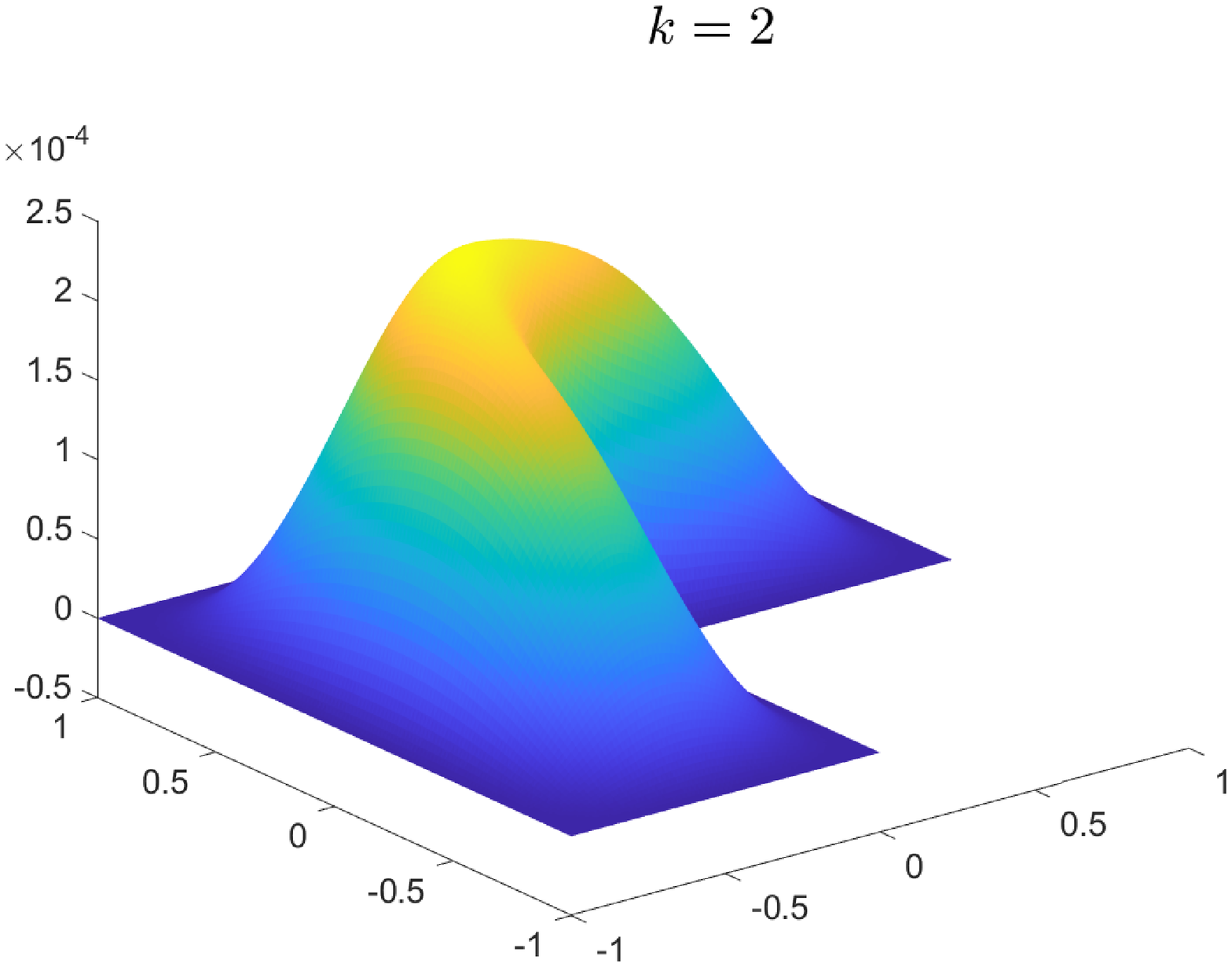}}
		\end{minipage}
		\caption{Discrete reconstructed field $f_{10^{-5},L}$ using \BFS (left) and Morley (right) FEMs, $k=0$ for Square domain and $k=2$ for L-shaped domain}\label{fig.squareLshapedmh}
	\end{center}
\end{figure}
\subsection{Conclusion}
The numerical results for the \BFS and Morley FEMs in the inverse problem are presented for square domain and L-shaped domain in Sections \ref{sec.numericalresultsBFS} and \ref{sec.numericalresultsMorley}. The outputs obtained for the square domain confirm the theoretical rates of convergence given in Theorems \ref{thm.forwarderror} and \ref{thm.totalerror} for $k=0$ and $s=r=\ell=4$. For the L-shaped domain, we expect reduced convergence rates for $m$ in $L^2$ norm and $f_\a$ in $H^k$ norm from the elliptic regularity. However, superconvergence results are obtained which indicates that the numerical performance is carried out in the non-asymptotic region for L-shaped domain which has corner singularity.

	\medskip
{\bf{Acknowledgements.}} The first author thanks National Board for Higher Mathematics, India for the financial support towards the research work (No: 0204/3/2020/R$\&$D-II/2476). 
%%%%%%%%%%%%%%%%%%%%%%%%%%%%%%%%%%%%%%%%%%%
\bibliographystyle{abbrv}
\bibliography{Fourth_order_elliptic}
\end{document}